\documentclass[12pt,]{article}
\usepackage[left=1in,right=1in,top=1in,bottom=1in,footskip=.25in]{geometry}

\usepackage{amsmath}
\usepackage{amssymb}
\usepackage{stackrel}
\usepackage{graphicx}
\usepackage{epstopdf}
\usepackage{bbm}
\usepackage{geometry}
\usepackage{enumerate}
\usepackage{etoolbox}
\usepackage[ruled,lined,linesnumbered,vlined]{algorithm2e}
\usepackage{tikz}

\usepackage{ThmsAndLemmas}
\usepackage{defs}
\usepackage{prjohns2MacrosTA}
\usepackage{jemacros}

\usepackage{color}

\usepackage{url}
\newcommand{\ualpha}{\underline{\alpha}}
\newcommand{\bw}{{\bf w}}

\title{Convergence Rates for Projective Splitting}

\author{Patrick R. Johnstone\thanks{Department of Management Science 
                                    and Information Systems, Rutgers Business
                                    School Newark and New Brunswick, Rutgers University}
        \and 
        Jonathan Eckstein$^*$}

\newcommand{\siamversion}[1]{}

\begin{document}
	\allowdisplaybreaks
\maketitle
%!TEX root = convratesv3.tex
\begin{abstract}
Projective splitting is a family of methods for solving inclusions involving
sums of maximal monotone operators. First introduced by Eckstein and Svaiter
in 2008, these methods have enjoyed significant innovation in recent years,
becoming one of the most flexible operator splitting frameworks  available.
While weak convergence of the iterates to a solution has been established,
there have been few attempts to study convergence rates of projective
splitting. The purpose of this paper is to do so under various
assumptions. To this end, there are three main contributions. First, in the
context of convex optimization, we establish an $O(1/k)$ ergodic function
convergence rate. Second, for strongly monotone inclusions, strong convergence
is established as well as an ergodic $O(1/\sqrt{k})$ convergence rate for the
distance of the iterates to the solution. Finally, for inclusions featuring
strong monotonicity and cocoercivity, linear convergence is established.
\end{abstract}

%!TEX root = convratesv3.tex
\section{Introduction}
For a real Hilbert space $\calH$, consider the problem of finding $z\in\calH$
such that 
\begin{align}\label{probCompMonoMany}
0\in \sum_{i=1}^{n} T_i z
\end{align}
where $T_i:\calH \to 2^{\calH}$ are maximal monotone
operators and additionally there exists a subset $\Iforw\subseteq\{1,\ldots,n\}$ such that for all $i\in\Iforw$ the operator $T_i$ is Lipschitz continuous.
An important instance of this problem is
\begin{align}\label{ProbOpt0}
\min_{z\in\mathcal{H}} F(z), \qquad\text{where~}\quad F(z)= \sum_{i=1}^{n}f_i(z)
\end{align}
and every $f_i:\calH\to(-\infty,+\infty]$ is closed, proper, and convex, with
some subset of the functions also being Fr\'echet differentiable with
Lipschitz-continuous gradients. Under appropriate constraint qualifications,
\eqref{probCompMonoMany} and
\eqref{ProbOpt0} are equivalent.  Problem~\eqref{ProbOpt0} arises in a host of
applications such as machine learning, signal and image processing, inverse
problems, and computer vision; 
% to name only a few (see
see~\cite{boyd2011distributed,combettes2011proximal,combettes2005signal} for some examples. 

A relatively recently proposed class of operator splitting algorithms which
can solve \eqref{probCompMonoMany}, among other problems, is \emph{projective
splitting}. It originated with~\cite{eckstein2008family} and was then
generalized to more than two operators in~\cite{eckstein2009general}.  The
related algorithm in~\cite{alotaibi2014solving} introduced a technique for
handling compositions of linear and monotone operators,
and~\cite{combettes2016async} proposed an extension to ``block-iterative'' and
asynchronous operation --- block-iterative operation meaning that only a
subset of the operators making up the problem need to be considered at each
iteration (this approach may be called ``incremental" in the optimization
literature).  A restricted and simplified version of this framework appears
in~\cite{eckstein2017simplified}. The recent work
\cite{johnstone2018projective} incorporated forward steps into the projective
splitting framework for any Lipschitz continuous operators and introduced
backtracking and adaptive stepsize rules.

In general, projective splitting offers unprecedented levels of flexibility
compared with previous operator splitting algorithms 
(\emph{e.g.}~\cite{mercier1979lectures,lions1979splitting,
tseng2000modified,davis2015three}). The
framework can be applied to arbitary sums of maximal monotone operators, the stepsizes can vary by operator and
by iteration, compositions with linear operators can be handled, and
block-iterative asynchronous implementations have been demonstrated.

In previous works on projective splitting, the main theoretical goal was
to establish the weak convergence of the iterates to a solution of the
monotone inclusion under study (either a special case or a generalization of
\eqref{probCompMonoMany}). This goal was achieved using Fej\'er-monotonicity
arguments in coordination with the unique properties of projective splitting.
The question of convergence rates has not been addressed, with the sole
exception of~\cite{machado2017complexity}, which considers a different type of
convergence rate than those investigated here; we discuss the differences
between our analysis and that of~\cite{machado2017complexity} in more detail
below.

% Since convergence rates
% are of critical importance from both a theoretical, and practical,
% perspective, it is important to fill this gap.

\vspace{0.2cm}   
\noindent 
{\bf Contributions}
\vspace{0.2cm}

\noindent 
 To this end, there are four main novel contributions in this paper. 
 \begin{enumerate}
 	\item For \eqref{ProbOpt0}, we establish an ergodic $\bigO(1/k)$ function
 	value convergence rate for  iterates generated by projective splitting.
 	\item When one of the operators in \eqref{probCompMonoMany} is strongly
 	monotone, we establish strong convergence, rather than weak, in the
 	general Hilbert space setting, without using the
 	Haugazeau~\cite{haugazeau68} modification employed to obtain general
 	strong convergence in~\cite{combettes2016async}. Furthermore, we derive an
 	ergodic $\bigO(1/\sqrt{k})$ convergence rate for the distance of  the iterates
 	to the unique solution of \eqref{probCompMonoMany}.
 	\item If additionally $T_1,\ldots,T_{n-1}$ are cocoercive, we establish
 	\emph{linear} convergence to $0$ of the distance of  the iterates to the
 	unique solution.
 	\item We discuss the special cases of projective splitting when $n=1$.
 	Interestingly, projective splitting reduces to one of two well-known
 	algorithms in this case depending on whether forward or backward steps are
 	used. This observation has implications for the convergence rate analysis.
 \end{enumerate}

The primary engine of the analysis is a new summability lemma
(Lemma~\ref{lemSum} below) in which important quantities of the algorithm are
shown to be summable. This summability is directly exploited in the ergodic
function value rate analysis in Section \ref{secFunc}. In
Section~\ref{secLin}, the same lemma is used to show linear convergence when
strong monotonicity and cocoercivity are present. With only strong
monotonicity present, we also obtain strong convergence and rates using a
novel analysis in Section \ref{secStrong}.

Our convergence rates apply directly to the variants of projective splitting
discussed in
\cite{combettes2016async,eckstein2008family,johnstone2018projective}. The
papers \cite{eckstein2017simplified,eckstein2009general} use a slightly
different separating hyperplane formulation than ours but the difference is
easy to resolve. However, our analysis does not allow for the asynchrony or
block-iterative effects developed in
\cite{combettes2016async,eckstein2017simplified,johnstone2018projective}. In
particular, at each iteration we assume that every operator is processed and
that the computations use the most up-to-date information.
% While it is possible to relax these assumptions in our analysis, doing so complicates
% our arguments and leads to significantly worse constants in the convergence rates.
Developing a convergence rate analysis which extends to asynchronous and
block-iterative operation in a satisfactory manner is a matter for future
work.
% Incorporating these features into a convergence rate analysis in a 
% satisfying way is a matter for future work. 

In \cite{combettes2016async,eckstein2017simplified,johnstone2018projective},
projective splitting was extended to handle the composition of each $T_i$ with
a linear operator. While it is possible to extend all of our convergence rate
results to allow for this generalization under appropriate conditions, for the
sake of readability we will not do so here.

In Section~\ref{secnis1} we consider the case $n=1$. In this case, we show
that projective splitting reduces to the proximal point method~\cite{Roc76a}
if one uses backward steps, or to a special case of the extragradient method
(with no constraint)
\cite{korpelevich1977extragradient,nguyen2018extragradient} when one uses
forward steps. Since projective splitting includes the proximal point method
as a special case, the $\bigO(1/k)$ function value convergence rate derived in
Section~\ref{secFunc} cannot be improved, since this is the best rate
available for the proximal point method, as established
in~\cite{guler1991convergence}.

The specific outline for the paper is as follows. Immediately below, we
discuss in more detail the convergence rate analysis for projective splitting
conducted in \cite{machado2017complexity} and how it differs from our
analysis. Section \ref{secNote} presents notation, basic mathematical
results, and assumptions. Section \ref{secAlgo} introduces the projective
splitting framework under study, along with some associated assumptions. Section \ref{secOld} recalls some important lemmas from
\cite{johnstone2018projective}. Section \ref{secNew} proves some new lemmas
necessary for convergence rate results, including the key summability lemma
(Lemma \ref{lemSum}). Section \ref{secFunc} derives the ergodic $\bigO(1/k)$
function value convergence rate for \eqref{ProbOpt0}. Section \ref{secStrong}
derives strong convergence and convergence rates under strong monotonicity.
Section \ref{secLin} establishes linear convergence under strong monotonicity
and cocoercivity. Finally, Section \ref{secnis1} discusses special cases of
projective splitting when $n=1$.

\vspace{0.2cm}   
\noindent 
{\bf Comparison with \cite{machado2017complexity}}
\vspace{0.2cm}

\noindent 
%\subsection{Comparison with \cite{machado2017complexity}}
%\label{secComp}
To the best of our knowledge, the only works attempting to quantify convergence
rates of projective splitting are \cite{machado2016projective} and
\cite{machado2017complexity}, two works by the same author.  The analysis in 
\cite{machado2016projective} concerns a dual application of
projective splitting and its convergence rate results are similar to those
in~\cite{machado2017complexity}.
%The analysis of \cite{machado2017complexity} also incorporates 
%$\epsilon$-enlarged monotone operators along the lines of 
%\cite{burachik1997enlargement} which allows for the use of 
%approximate solvers in the subproblems.
In these works, convergence rates are not defined in the more customary way they
are in this paper --- either in terms of the distance to the solution or the
gap between current function values and the optimal value of \eqref{ProbOpt0}.
Instead in \cite{machado2017complexity} they are defined in terms of an
approximate solution criterion for the monotone inclusion under study, specifically
\eqref{probCompMonoMany} with $n=2$. Without any enlargement being applied to the
operators, the approximate solution condition is as follows:  a point $(x,y)\in
\calH^2$ is said to be an $\epsilon$-approximate solution
of~\eqref{probCompMonoMany} with $n=2$ if there exists $(a,b)\in\calH^2$ s.t.
$a\in T_1 x$, $b\in T_2 y$ and
$
\max\{\|a+b\|,\|x-y\|\}\leq\epsilon.
$
If this condition holds with $\epsilon=0$, then $x=y$ is a solution
to~\eqref{probCompMonoMany}. The iteration complexity of a method is then
defined in terms of the number iterations required to produce a point $(x^k,y^k)$
which is a $\epsilon$-approximate solution in this sense.

We stress that this notion of iteration complexity/convergence rate is
different to what we use here. Instead, we directly analyze the distance of
the points generated by the algorithm to a solution of
\eqref{probCompMonoMany}, that is,  we study the behavior of $\|z^k - z^*\|$,
where $z^*$  solves
\eqref{probCompMonoMany}. Or for the special case of \eqref{ProbOpt0}, we
consider the convergence rate of $F(z^k)-F^*$ to zero, where $F^*$ is
the optimal value. 

% Under appropriate assumptions we  determine convergence
% rates of these quantities to $0$ which in turn gives one the number of
% iterations required to drive each quantity below a given target accuracy.

%!TEX root = convratesv3.tex
\section{Mathematical Preliminaries}\label{secNote}
\subsection{Notation, Terminology, and Basic Lemmas} \label{subsecNote}
Throughout this paper we will use the standard convention that a sum over an
empty set of indices, such as $\sum_{i=1}^{n-1}a_i$ with $n=1$, is taken to be
$0$.
%
% Summations of the form:  for some collection $\{a_i\}$ will appear throughout
% the paper. To deal with the case: $n=1$, throughout the paper use the standard
% definition
% \begin{align}\label{eqn1}
% \sum_{i=1}^{0} a_i = 0.
% \end{align}
%
%  A monotone operator $A$ satisfies 
%  \begin{eqnarray*}
%  \langle u - v,u' - v'\rangle \geq 0,\forall u,v,u'\in A(u),v'\in A(v).
%  \end{eqnarray*}
% Define the graph of $A$ as
%  \begin{eqnarray*}
% 	\gra A = \{(x,y):y\in A(x)\}.
% \end{eqnarray*}
%  A monotone operator is maximal if its graph is not a subset of the graph of any other monotone operator. 
% A strongly monotone operator $A$ satisfies
% \begin{eqnarray*}
% \langle u - v,u' - v'\rangle \geq \mu\|u - v\|^2_2,
% \quad\forall u,v,u'\in A(u),v'\in A(v)
% \end{eqnarray*}
% for some $\mu>0$. 
A single-valued operator $A: \calH \rightarrow \calH$ is called \emph{cocoercive} if, for some $\Gamma>0$,
\[
(\forall\,u,v\in\calH) \qquad
 \langle u - v,A(u) - A(v)\rangle \geq \Gamma^{-1} \|A(u) - A(v)\|^2_2.
\]
Every cocoercive operator is $\Gamma$-Lipschitz
continuous, but not \emph{vice versa}.
%We use the generalized notation for the prox operator introduced in \cite{eckstein2017simplified}. Specifically
%\begin{eqnarray*}
%(x,y) = \Prox_A^\mu(z)\implies
%x+\mu y = z,\text{ and }y\in A(x). 
%\end{eqnarray*}
%
For any maximal monotone operator $A : \calH \rightarrow 2^\calH$ and scalar
$\rho > 0$ we will use the notation
%\begin{eqnarray*}
$
	\prox_{\rho A} = (I+\rho A)^{-1}
$
%\end{eqnarray*}
to denote the \emph{proximal operator}, 
also known as the backward or implicit step with respect to $A$.  In particular,
\begin{eqnarray}\label{defprox2}
x = \prox_{\rho A}(a) \quad\implies\quad \exists y\in Ax:x+\rho y = a,
\end{eqnarray}
and the $x$ and $y$ satisfying this relation are unique. 
Furthermore, $\prox_{\rho A}$ is defined everywhere and
$\range(\prox_A) = \text{dom}(A)$ \cite[Proposition 23.2]{bauschke2011convex}. 
In the special case where $A = \partial f$ for some convex, closed, and
proper function $f$, the proximal operator may be written as
\begin{align}
\label{defProxOpt}
\prox_{\rho \partial f}
=
\argmin_x\left\{\frac{1}{2}\|x-a\|^2 + \rho f(x)\right\}.
\end{align}
In the optimization context, we will use the notational convention that
$\prox_{\rho\partial f} = \prox_{\rho f}$.
%
% For closed, convex, and proper functions $f$, we will use the subgradient
% inequality \cite[Definition 16.1]{bauschke2011convex}
% \begin{align}\label{eqsubg}
% (\forall x,y: g\in\partial f(y))\quad 
% f(y)-f(x) \leq \langle g,y-x\rangle.
% \end{align}
% %If $f$ is closed, convex, proper, and Fr\'echet differentiable with $L$-Lipschitz continuous gradient, then we will use the descent property 
% %\cite[Theorem 18.15(iii)]{bauschke2011convex}
% %\begin{align}
% %\label{eqDescent}
% %(\forall x,y\in\calH)\quad f(y)-f(x)\leq \langle\nabla f(x),y-x\rangle + \frac{L}{2}\|y-x\|^2.
% %\end{align}
%
Finally, we will use the following two standard results:
\begin{lemma}\label{lemBasic}
	For any vectors $v_1,\ldots,v_n\in\calH^n$, 
	$
	\left\|\sum_{i=1}^n v_i\right\|^2\leq n\sum_{i=1}^n\left\|v_i\right\|^2.
	$
\end{lemma}
\begin{lemma}
	For any $x,y,z\in\calH$
\begin{align}
\label{eqCosine}
2\langle x-y,x-z\rangle = 
\|x-y\|^2+\|x-z\|^2 - \|y-z\|^2.
\end{align}
\end{lemma} 
\noindent We will use a boldface $\bw = (w_1,\ldots,w_{n-1})$ for elements 
of $\calH^{n-1}$.

\subsection{Main Assumptions Regarding Problem (\ref{probCompMonoMany})}\label{secMainAss}
%In closing this section, we state our main assumption regarding~(\ref{probCompMonoMany}). 
%Let $\boldsymbol{\mathcal{H}} = \mathcal{H}^{n}$.
Define the \emph{extended solution set} or \emph{Kuhn-Tucker set}
of~\eqref{probCompMonoMany} to be
% $\calS \subseteq\mathcal{C}$ s.t.
\begin{equation} \label{defCompExtSol}
\calS
=  \Big\{ (z,w_1,\ldots,w_{n-1}) \in \calH^n \;\; \Big| \;\;
w_i\in T_i z,\,\, i=1,\ldots,n-1, \;\;
%\text{ and }
-\textstyle{\sum_{i=1}^{n-1}  w_i} \in T_n z \Big\}.
\end{equation}
% \begin{eqnarray}
% \calS  = 
% \{
% (w_1,\ldots,w_{n-1},z)
% \,\,\text{ s.t. }
% w_i&\in& A_i G_i z,\quad i=1,\ldots,n,
% \nonumber\\\nonumber
% w_{n+j} &=& B_j G_{n+j} z,\quad j=1,\ldots,m-1,
% \\
% 0&=& 
% \sum_{i=1}^{n-1} G_i^* w_i 
% +B_{n}z
% \}.
%\end{eqnarray}
Clearly $z\in\mathcal{H}$ solves~\eqref{probCompMonoMany} if and only if
there exists $\bw\in\calH^{n-1}$ such that
$(z,\bw)\in\calS$. 
Our main assumptions regarding~\eqref{probCompMonoMany} are
as follows:
\begin{assumption}
	\label{AssMonoProb}\label{assMono}	
	$\mathcal{H}$ is a real Hilbert space and
	problem~\eqref{probCompMonoMany} conforms to the following:
	\begin{enumerate}
	\item For $i=1,\ldots,n$, the operators
	$T_i:\mathcal{H}\to2^{\mathcal{H}}$ are monotone. 
	\item For all $i$ in some subset $\Iforw \subseteq \{1,\ldots,n\}$,
    the operator $T_i$ is
	$L_i$-Lipschitz continuous (and thus single-valued)
	and $\text{dom}(T_i) = \mathcal{H}$.  
	\item For $i \in
	\Iback \triangleq \{1,\ldots,n\} \backslash \Iforw$, the operator 
	$T_i$ is maximal and that the map $\prox_{\rho
	T_i}:\mathcal{H}\to\mathcal{H}$ can be computed 
to within the error tolerance specified below in Assumption \ref{assErr} 
	(however, these operators are not precluded 
	from also being Lipschitz continuous).
	\item The solution set $\calS$ defined in
	(\ref{defCompExtSol}) is nonempty.
	\end{enumerate}
\end{assumption}

\begin{proposition} \emph{\cite[Lemma~3]{johnstone2018projective}} 
\label{prop:SClosedConv}
Under Assumption~\ref{assMono}, $\calS$ from~\eqref{defCompExtSol}
is closed and convex.
\end{proposition}

% Lipschitz continuous operators may still be handled via their proximal
% mappings if that is convenient. Thus the set $\Iback$ may include operators
% that are Lipschitz continuous. A specific example of such a choice is
% discussed in the numerical experiments of Section~\ref{secNumerical}.

%!TEX root = convratesv3.tex
\section{The Algorithm}
\label{secAlgo}

%\subsection{General Separator-Projector Methods}
Projective splitting is a special case of a general seperator-projector method
for finding a point in a closed and convex set. At each iteration the method
constructs an affine function $\varphi_k:\calH^n\to \mathbb{R}$ which
separates the current point from the target set $\calS$ defined in
\eqref{defCompExtSol}. In other words, if $p^k$ is the current point in
$\calH^n$ generated by the algorithm, $\varphi_k(p^k)>0$, and
$\varphi_k(p)\leq 0$ for all $p\in\calS$. The next point is then the
projection of $p^k$ onto the hyperplane $\{p:\varphi_k(p)=0\}$, subject to a
relaxation factor $\beta_k$. What makes projective splitting an operator
splitting method is that the hyperplane is constructed through individual
calculations on each operator $T_i$, either $\prox$ calculations or forward
steps.

\subsection{The Hyperplane}\label{secTecLems}\label{secHplane}
Let $p =(z,\bw) =  (z,w_1,\ldots, w_{n-1})$ be a
generic point in $\calH^n$. For $\calH^n$, we
adopt the following norm and inner product for some $\gamma>0$:
\begin{align} \label{gammanorm}
\norm{(z,\bw)}^2 &= \gamma\|z\|^2 + \sum_{i=1}^{n-1}\|w_i\|^2 &
\Inner{(z^1,\bw^1)}{(z^2,\bw^2)}_\gamma &= 
\gamma\langle z^1,z^2\rangle + \sum_{i=1}^{n-1}\langle
w^1_i,w^2_i\rangle.
\end{align}
Define the following function
for all $k\geq 1$:
\begin{eqnarray}\label{hyper}\label{hplane}\label{defHyper}
\varphi_k(p)
&=&
\sum_{i=1}^{n-1}
\left
\langle  z - x_i^k,y_i^k - w_i
\right
\rangle
+
\left\langle
z-x_{n}^k,
y_{n}^k
+
\sum_{i=1}^{n-1}  w_i
\right\rangle.
\label{defCompHplane}
\end{eqnarray}
where the $(x_i^k,y_i^k)$ are chosen so that $y_i^k\in T_i x_i^k$ for
$i=1,\ldots,n$.
%(recall that each inner product is for the corresponding
%Hilbert space $\calH_i$).
%but for simplicity we have dropped the subscript.  
This function is a special case of the separator function used
in~\cite{combettes2016async}.  The following lemma proves some basic
properties  of $\varphi_k$; similar results are
in~\cite{alotaibi2014solving,combettes2016async,eckstein2017simplified} in the
case $\gamma=1$.

\begin{lemma}\label{LemGradAffine} 
\emph{\cite[Lemma 4]{johnstone2018projective}} Let $\varphi_k$ be defined as in (\ref{hplane}).  Then:
	\begin{enumerate}
		\item $\varphi_k$ is affine on $\calH^n$. 
		\item With respect to inner product~\eqref{gammanorm} defined
		on $\calH^n$, the gradient
		of $\varphi_k$ is
		\begin{align}\label{defGrad}
		\nabla\varphi_k =
		\left(\frac{1}{\gamma}\left(\sum_{i=1}^{n}y_i^k\right) \!\!,\;
		x_1^k -  x_{n}^k,x_2^k - x_{n}^k,\ldots,x_{n-1}^k -  x_{n}^k\right).		
		\end{align}
		\item 
		Suppose Assumption \ref{AssMonoProb} holds and 
		that $y_i^k\in T_i x_i^k$ for $i=1,\ldots,n$. Then
		$\varphi_k(p)\leq 0$ for all $p\in \calS$ defined in (\ref{defCompExtSol}).
		\item \label{lem0grad}
		If Assumption \ref{AssMonoProb} holds, $y_i^k\in T_i x_i^k$, and $\nabla\varphi_k =  0$, then 
		$
		(x_n^k,y_1^k,\ldots,y_{n-1}^k)\in\calS.
		$
	\end{enumerate}
\end{lemma}

%If we use variable $w_n^k$ defined in \eqref{defWnm}, the affine function evaluated at $p^k$ can be written as
%\begin{align}\label{defHplane2}
%\varphi_k(p^k)
%&=
%\sum_{i=1}^{n}
%\left
%\langle  z^k - x_i^k,y_i^k - w_i^k
%\right
%\rangle
%\end{align}

\subsection{Projective Splitting} 

Algorithm~\ref{AlgfullyAsync} is the projective splitting framework for which
we will derive convergence rates. It is a special case of the framework
of~\cite{johnstone2018projective} without asynchrony or block-iterative
features. In particular, we assume at each iteration that the method processes
every operator $T_i$, using the most up-to-date information possible.
%In contrast,
%\cite{combettes2016async,eckstein2017simplified,johnstone2018projective} all
%relax these conditions. Deriving convergence rates in the presence of these
%features is left as a matter for future work.
The frameworks of
\cite{combettes2016async,eckstein2017simplified,johnstone2018projective} also
allow for the monotone operators to be composed with linear operators. As
mentioned in the introduction, our analysis may be extended to this situation, but
for the sake of readibility we will not do so here.

Algorithm \ref{AlgfullyAsync} is a special case of the separator-projector
algorithm applied to finding a point in $\calS$ using the affine function
$\varphi_k$ defined in \eqref{hyper} \cite[Lemma 6]{johnstone2018projective}.
For the variables of Algorithm \ref{AlgfullyAsync} defined on lines
\ref{eqAlgproj1}--\ref{eqAlgproj2}, define $p^k = (z^k,\bw^k)\transpose =
(z^k,w_1^k,\ldots,w_{n-1}^k)\transpose$ for all $k\geq 1$. The points
$(x_i^k,y_i^k)\in\gra T_i$ are chosen so that $\varphi_k(p^k)$ is sufficiently
large to guarantee the weak convergence of $p^k$ to a solution
\cite[Theorem~1]{johnstone2018projective}. For $i\in\Iback$, a single $\prox$
calculation is required, whereas for $i\in\Iforw$, two forward steps are
required per iteration.

The algorithm has the following parameters:
\begin{itemize}
	%\item An integer $\hat{k}\geq 1$. For iterations before %$\hat{k}$ we allow
	%the algorithm to arbitrarily set the points $(x_i^k,y_i^k)$.  This allows
	%us to avoid a potentially costly intialization routine. 
	\item For each $k\geq 1$ and $i=1,\ldots,n$, a positive scalar 
	stepsize $\rho_{i}^{k}$.
	\item For each $k\geq 1$, an overrelaxation parameter
	$\beta_k\in[\underline{\beta},\overline{\beta}]$ where
	$0<\underline{\beta}\leq\overline{\beta}<2$.
	\item The fixed scalar $\gamma>0$ from~\eqref{gammanorm}, which controls
	the relative emphasis on the primal and dual variables in the projection
	update in lines \ref{eqAlgproj1}-\ref{eqAlgproj2}.	
	% We use the norm $\|p\|_\gamma^2 = \|(z,{\bf w})\|_\gamma^2 =
	% \gamma\|z\|^2+\sum_{i=1}^{n-1} \|w_i\|^2$ in the primal-dual space
	% $\boldsymbol{\mathcal{H}}$ defined in Section~\ref{secMainAss};
	%\item $\text{Flag}_{ss}$ which equals ``True" if the sufficiently-large
	%stepsize condition is to be checked and ``False" otherwise. This check is
	%needed to derive convergence rates, although it isn't needed for overall
	%convergence.
	\item Sequences of errors $\{e_i^k\}_{k\geq 1}$ for $i\in\Iback$, modeling
	inexact computation of the proximal steps.
\end{itemize}

To ease the mathematical presentation, we use the following notation in
Algorithm \ref{AlgfullyAsync} and throughout the rest of the paper:
\begin{align}
\label{defWnm}\
(\forall\,k\in\mathbb{N}) \quad\quad w_{n}^k &\triangleq -\sum_{i=1}^{n-1} w_i^k.
\end{align}
Note that when $n=1$, we have $w_n^k = 0$ by the convention at the start of
Section~\ref{subsecNote}.

\begin{algorithm}[t]
	\DontPrintSemicolon
	\SetKwInOut{Input}{Input}
	\Input{$(z^1,{\bf w}^1)\in \calH^n$, $\gamma>0$.}
	\For{$k=1,2,\ldots$}
	{   
		%\If{$k\geq \hat{k}$}
		%{
		\For{$i=1,2,\ldots, n$}
		{
			
			\If{$i\in\Iback$}
			{
				$a =  z^{k}+\rho_{i}^{k} w_i^{k}+e_i^k$\label{lineaupdate}\;
				$x_i^k = \prox_{\rho_{i}^{k} T_i}(a)
				$\label{LinebackwardUpdate}\;
				$
				y_i^k = (\rho_{i}^{k})^{-1}
				\left(
				a - x_i^k
				\right)
				$\label{lineBackwardUpdateY}\;

			}
			\Else
			{\label{LineForwardUpdate}
				$					
				x_i^k =  z^{k}-\rho_{i}^{k}
				( T_i  z^{k} - w_{i}^{k}),
				$\label{ForwardxUpdate}\;            	
				$
				y_i^k = T_i x_i^k.
				$
				\label{ForwardyUpdate}\;
				
			}
			
		}
		%}
		%\Else
		%{
		%  $(x_i^k,y_i^k)\in\mathcal{H}_i^2$ may be chosen arbitrarily for $i=1,\ldots,n$.	
		%} 

	   $u_i^k = x_i^k -  x_n^k,\quad i=1,\ldots,n-1,$\label{lineCoordStart}\;
		$v^k = \sum_{i=1}^{n}  y_i^k$\label{lineVupdate}\;    
		$\pi_k = \|u^k\|^2+\gamma^{-1}\|v^k\|^2$  \label{linePiUpdate}\;    
		\eIf{$\pi_k>0$}{
			\label{lineHplane} 
			$\varphi_k(p_k) = 
			\langle z^k, v^k\rangle 
			+
			\sum_{i=1}^{n-1}
			\langle w_i^k,u_{i}^k\rangle 
			-
			\sum_{i=1}^{n}
			\langle x_i^k,y_i^k\rangle  
			$\label{lineComputeHplane}\;
			$\alpha_k = {\beta_k\varphi_k(p_k)}/{\pi_k}
			$\label{linealpha}\;
		}
		{
			\Return $z^{k+1}\leftarrow x_n^k, w_1^{k+1}\leftarrow y_1^k,\ldots,w_{n-1}^{k+1}\leftarrow y_{n-1}^k$\label{lineReturn}
		} 
		$z^{k+1} = z^k - \gamma^{-1}\alpha_k v^k$\label{eqAlgproj1} \;
		$w_i^{k+1} = w_i^k - \alpha_k u_{i}^k,\quad i=1,\ldots,n-1$,\label{eqAlgproj2} \;
		$w_{n}^{k+1} = -\sum_{i=1}^{n-1}  w_{i}^{k+1}$\label{lineCoordEnd}\;
	}
	\caption{Algorithm for solving~\eqref{probCompMonoMany}.}
	\label{AlgfullyAsync}
\end{algorithm}

%\input{hyperplane}
%!TEX root = convratesv3.tex
\subsection{Conditions on the Errors and the Stepsizes}
We now state our assumptions regarding the computational errors and stepsizes
in Algorithm~\ref{AlgfullyAsync}. Assumptions \eqref{err1} and \eqref{err2}
are taken exactly from \cite{eckstein2017simplified}. Assumption \eqref{err3}
is new and necessary to derive our convergence rate results. Throughout the
rest of the manuscript, let $\bar{K}$ be the iteration where Algorithm
\ref{AlgfullyAsync} terminates via Line \ref{lineReturn}, with
$\bar{K}=\infty$ if the method runs indefinitely.

\begin{assumption}\label{assErr}
For some $\sigma \in[0,1[$ 
	and $\delta\geq 0$, the following hold for all
	$1\leq k \leq\bar{K}$ and $i\in\Iback$:
\begin{align}
	\langle  z^{k} - x_i^k,e_i^{k}\rangle
	&\geq
	-\sigma\| z^{k} - x_i^k\|^2
	\label{err1}\\
	\langle e_i^{k},y_i^k - w_i^{k}\rangle 
	&\leq \rho^{k}_i\sigma \|y_i^k - w_i^{k}\|^2
	\label{err2}
	\\\label{err3}
\|e_i^{k}\|^2 &\leq \delta\| z^{k} - x_i^k\|^2.
\end{align}
\end{assumption}
\noindent Note that if \eqref{err3} holds with $\delta<1$, then 
\eqref{err1} holds for any $\sigma \in [\delta,1[$.  For each $i\in\Iback$,
we will show in~\eqref{eqLast1} below that the sequence $\{\| z^{k} - x_i^k\|
\}$ is square-summable, so an eventual consequence of
\eqref{err3} will be that the corresponding error sequence must be
square-summable, that is,
$
\sum_{k=1}^{\bar{K}}\|e_i^k\|^2<\infty.
$

\begin{assumption}\label{assStep}
The stepsizes satisfy
\begin{align}
\underline{\rho} &\triangleq 
    \min_{i=1,\ldots,n} \left\{ \inf_{1\leq k\leq\bar{K}} \rho_i^k\right\} > 0 
\label{steplow}\\\label{stephigh}
(\forall i\in\Iback)\quad 
\overline{\rho}_i &\triangleq
    \sup_{1\leq k\leq\bar{K}} \rho_i^k  < \infty 
    \\
\label{stepforlow}
(\forall\,i\in\Iforw)\quad
\overline{\rho}_i &\triangleq 
\sup_{1\leq k\leq\bar{K}} \rho_i^k < \frac{1}{L_i}.
\end{align}
%and
%\begin{eqnarray}\label{LemSumParams1}
%\frac{\prod_{i\in\Iforw^+}\overline{\rho}_i}
%{|\Iforw^+|\underline{\rho}^{|\Iforw^+|-1}}&\leq&
%\frac{1-c}{2L_s}
%\end{eqnarray}
%where $0<c<1$, $L_s = \sum_{i\in\Iforw} L_i$, $\Iforw^+ = \Iforw\cup\{n\}$,
%and if $n\in\Iback$, $\overline{\rho}_n = \overline{\rho}$.
\end{assumption}
\noindent From here on, we let 
$\orho = \max_{i\in\Iback}\orho_i$ and $\bar{L} = \max_{i\in\Iforw} L_i$ with the convention that $\orho = 0$ if $\Iback = \{\emptyset\}$ and $\bar{L} = 0$  if $\Iforw = \{\emptyset\}$.  

The recent work \cite{johnstone2018projective} includes several extensions to
the basic forward-step stepsize constraint \eqref{stepforlow}, under which one
still obtains weak convergence of the iterates to a solution.  Section 4.1 of
\cite{johnstone2018projective} presents a backtracking linesearch that may be
used when the constant $L_i$ is unknown. Section 4.2 of
\cite{johnstone2018projective}  presents a backtrack-free
adaptive stepsize for the special case in which $T_i$ is affine but $L_i$ is
unknown. Our convergence rate analysis holds for these extensions, but for the
sake of readability we omit the details.

%!TEX root = convratesv3.tex
\newcommand{\bcalH}{\boldsymbol{\mathcal{H}}}
\section{Results Similar to \cite{johnstone2018projective}}
\label{secOld}
% In this section, we review some important results about Algorithm
% \ref{AlgfullyAsync} from \cite{johnstone2018projective}. In the first we
% rewrite the projection updates on lines \ref{eqAlgproj1}--\ref{eqAlgproj2} and
% specify the Fej\'er monotonicity properties of the algorithm.

\begin{lemma} 
Suppose Assumption \ref{assMono} holds. In the context of Algorithm \ref{AlgfullyAsync},
recall the notation $p^k = (z^k,w_1^k,\ldots,w_{n-1}^k)\transpose$. For all $1\leq
k<\bar{K}$:
	\begin{enumerate}
		\item The iterates can be written as
		\begin{align}\label{eqUpds}
		p^{k+1} = p^k - \alpha_k\nabla\varphi_k(p^k)
		\end{align}
		where $\alpha_k = {\beta_k\varphi_k(p^k)}/{\|\nabla\varphi_k\|}$
		and $\varphi_k$ is defined in \eqref{hyper}.
		\item For all $p^*=(z^*,w_1^*,\ldots,w_{n-1}^*)\in\calS$:
		\begin{align}\label{eqFejer}
		\|p^{k+1}-p^*\|^2\leq\|p^k - p^*\|^2 - \beta_k(2-\beta_k)\|p^{k+1}-p^k\|^2
		\end{align}
		which implies that
		\begin{align}\label{eqSums}
		\sum_{t=1}^k\|p^t - p^{t+1}\|^2 &\leq \tau\|p^1 - p^*\|^2,
		\quad\quad\quad\quad\quad\quad\quad\quad\quad
		\text{where~}
		&&
		\tau = \ubeta^{-1}(2-\obeta)^{-1} \\
		% \end{align}
  %       and
		% \begin{align}\label{zbegin}
		\|z^{k} - z^*\|&\leq\gamma^{-\frac{1}{2}}\|p^k - p^*\|\leq\gamma^{-\frac{1}{2}}\|p^1-p^*\|,
		\\\label{wbegin}
		\|w_i^{k} - w_i^*\|&\leq\|p^k - p^*\|\leq\|p^1-p^*\|, &&i=1,\ldots,n-1.
		\end{align} 
	\end{enumerate}\label{propFejer}
\end{lemma}
\begin{proof}
	The update \eqref{eqUpds} follows from algebraic manipulation of lines
	\ref{lineCoordStart}--\ref{lineCoordEnd} and consideration of
	\eqref{hplane} and \eqref{defGrad}.
	Inequalities \eqref{eqFejer} and \eqref{eqSums} result from 
Algorithm \ref{AlgfullyAsync} being a separator-projector
algorithm \cite[Lemma 6]{johnstone2018projective}. A specific reference
proving these results is \cite[Proposition 1]{eckstein2008family}.
%
%	While it has previously been shown that the sequences $\{p^k\}$,
%	$\{x_i^k\}$, and $\{y_i^k\}$ are bounded in \cite{eckstein2017simplified}
%	and \cite{johnstone2018projective}, explicit upper bounds have not yet
%	been derived. We do so in appendix \ref{appendBound}.
%
\end{proof} 

The following lemma places an upper bound on the gradient of the affine function
$\varphi_k$ at each iteration. A similar result was proved in \cite[Lemma
11]{johnstone2018projective}, but since that result is slightly different, we
include the full proof here.

% The following lemma finds two lower bounds for $\varphi_k(p^k)$. It was proved
% in \cite[lemmas 12 and 13]{johnstone2018projective} under more general
% conditions.

%!TEX root = convratesv3.tex
\begin{lemma}\label{boundedGradient}\label{lemBoundedGrad}
	Suppose assumptions~\ref{assMono},\ref{assErr}, and \ref{assStep} hold and recall the
	affine function $\varphi_k$ defined in~\eqref{hplane}.  For all $1\leq k\leq\bar{K}$,
$$
		\|\nabla\varphi_k\|^2
		\leq
\xi_1\sum_{i=1}^n\|z^k-x_i^k\|^2
$$
where 
\begin{align}\label{defxi1}
\xi_1 = 
2n\left[1+2\gamma^{-1}\left(\bar{L}^2|\Iforw|+\underline{\rho}^{-2}(1+\delta)\right)\right]<\infty.
\end{align} 
\end{lemma}
\begin{proof}
Using
	Lemma~\ref{LemGradAffine},
	\begin{eqnarray}
		\|\nabla \varphi_k\|^2
		=
	\gamma^{-1}\left\|\sum_{i=1}^{n}  y^k_i\right\|^2
	+
		\sum_{i=1}^{n-1}
		\|x^k_i -  x_{n}^k\|^2
    \label{eqGrad}.
	\end{eqnarray}
	Using Lemma~\ref{lemBasic}, we begin by writing the second term on the
	right of~\eqref{eqGrad} as
	\begin{align}
				\sum_{i=1}^{n-1}
		\|x^k_i -  x_{n}^k\|^2
		&\leq
2\sum_{i=1}^{n-1}
\left( 
	\|x^k_i -  z^k\|^2 +\| z^k -  x_{n}^k\|^2
	\right) 
\leq 
2n\sum_{i=1}^{n}
\|z^k-x^k_i\|^2 
. \label{secondterm}
	\end{align}

We next consider the first term in~\eqref{eqGrad}. Rearranging the update
equations for Algorithm~\ref{AlgfullyAsync} as given in lines \ref{LinebackwardUpdate} and \ref{ForwardxUpdate}, we may write
\pagebreak[3]
	\begin{align}\label{eqRearrange1}
		y_i^k &= \left(\rho_{i}^{k}\right)^{-1}
		\left( z^k - x_i^k+\rho_{i}^{k}w_i^{k}+e_i^{k}\right),
		& i &\in\Iback
		\\\label{eqRearrange2}
		T_i  z^{k} &=
		\left(\rho_{i}^{k}\right)^{-1}
		\left( z^{k} - x_i^k+\rho_{i}^{k}w_i^{k}\right),
		& i&\in\Iforw.
	\end{align}
	Note that \eqref{eqRearrange1} rewrites line \ref{LinebackwardUpdate}
of Algorithm~\ref{AlgfullyAsync} using \eqref{defprox2}. The first term
of~\eqref{eqGrad} may then be written as	% Substituting (\ref{eqRearrange1})--(\ref{eqRearrange2}) into the first term of (\ref{eqGrad}) and ignoring the scalar $\gamma^{-1}$ yields
	\begin{align}
    \left\|\sum_{i=1}^{n} y^k_i\right\|^2
	&=
	\left\|\sum_{i\in\Iback} y^k_i + \sum_{i\in\Iforw}  \left(T_i z^{k} + y^k_{i} - T_i z^{k}\right)\right\|^2
	\nonumber\\\nonumber
	&\overset{\text{(a)}}{\leq}
	2\left\|\sum_{i\in\Iback}  y^k_i + \sum_{i\in\Iforw}  T_i  z^{k}\right\|^2
	+
	2\left\|\sum_{i\in\Iforw} \left(y^k_{i} - T_i z^{k}\right)\right\|^2
	\\
	&\overset{\text{(b)}}{=}
	2\left\|\sum_{i=1}^n  \big(\rho_i^{k}\big)^{-1} 
	\left(  z^{k} - x_i^k+\rho_i^{k}w_i^{k}
	\right) + \sum_{i\in\Iback}\big(\rho_i^{k}\big)^{-1} e_i^{k}\right\|^2 \nonumber \\
	&\qquad\quad 
	+\; 2\left\|\sum_{i\in\Iforw} \left(Tx^k_{i} - T_i z^{k}\right)\right\|^2
	\nonumber \\
	&\overset{\text{(c)}}{\leq}
	4\left\|\sum_{i=1}^n \big(\rho_i^{k}\big)^{-1} 
	\left(  z^{k} - x_i^k+\rho_i^{k}w_i^{k}
	\right)\right\|^2
	+
	4\left\|\sum_{i\in\Iback}\big(\rho_i^{k}\big)^{-1} e_i^{k}\right\|^2
	\nonumber\\&\qquad\quad
	+\;2|\Iforw|\sum_{i\in\Iforw} 
	\left\|
	T_i x^k_{i} - T_i z^{k}\right\|^2
	\nonumber \\
	&\overset{\text{(d)}}{\leq}
	4n\underline{\rho}^{-2}
	\left(
	\sum_{i=1}^n \left\| 
	  z^{k} - x_i^k
	\right\|^2
	+
	\sum_{i\in\Iback}
	\|e_i^{k}\|^2
	\right) 
	+
	% \nonumber\\&\qquad
	% +\;
	2|\Iforw|\sum_{i\in\Iforw} \left(L_i^2\|x_{i}^k-z^{k}\|^2 \right)	
	\nonumber 
	\\
	&\overset{\text{(e)}}{\leq}
    \xi_1'\sum_{i=1}^n\|z^k-x_i^k\|^2   \label{eqEndBounded}
	\end{align}
	where 
	\begin{align*}
	\xi_1' = 
	4n(\bar{L}^2|\Iforw|+\underline{\rho}^{-2}(1+\delta))
	\end{align*}
	where recall $\bar{L} = \max_{i\in\Iforw}L_i$.
	In the above, (a) uses Lemma~\ref{lemBasic}, while (b) is obtained by substituting 
	\eqref{eqRearrange1}-\eqref{eqRearrange2} into the first squared norm
	and using $y_i^k =T_i x_i^k$ for $i\in\Iforw$ in the second. Next, (c)
	uses Lemma~\ref{lemBasic} once more on both terms. Inequality (d) uses
	Lemma~\ref{lemBasic}, the Lipschitz continuity of $T_i$, $\sum_{i=1}^n
	w_i^k = 0$, and Assumption~\ref{assStep}. Finally, (e) follows by collecting terms and using Assumption ~\ref{assErr}. Combining~\eqref{eqGrad},
	\eqref{secondterm}, and~\eqref{eqEndBounded} establishes the Lemma with
	$\xi_1$ as defined in \eqref{defxi1}.
\end{proof}	
%!TEX root = convratesv3.tex
\begin{lemma}
\label{PositivePhi}	\label{lemLowerBoundPhi}
	Suppose that assumptions~\ref{AssMonoProb}, \ref{assErr}, and \ref{assStep} hold.  
	Then for all $1\leq k\leq\bar{K} $
\begin{equation}\nonumber
\varphi_k(p^k)
	\geq
	\xi_2
	\sum_{i=1}^n
	\|z^{k}-x^k_i\|^2
.
\end{equation}
where
\begin{align}
\xi_2 = 
\min
\left\{
(1-\sigma)\overline{\rho}^{-1}
,\;
\min_{j\in\Iforw}
\left\{
\overline{\rho}_{j}^{-1} - L_j 
\right\}
\right\}>0\label{defxi2}.
\end{align}
Furthermore, for all such $k$,
\begin{align}\label{eqAltLB}
\varphi_k(p^k)+\sum_{i\in\Iforw}L_i\|z^k - x_i^k\|^2
\geq
(1-\sigma)\urho\sum_{i\in\Iback}\|y_i^k - w_i^k\|^2
+
\urho\sum_{i\in\Iforw}\|T_i z^k - w_i^k\|^2.
\end{align}
\end{lemma}
\begin{proof}
The claimed results are special cases of those given in lemmas 12-13
of~\cite{johnstone2018projective}.
\end{proof}
\section{New Lemmas Needed to Derive Convergence Rates}\label{secNew}
%!TEX root = convratesv3.tex
\begin{lemma}\label{LemEveryOtherIter}\label{lb_ratio}\label{lemma_BoundedStepsize}\label{lemStepBounded}
	Suppose assumptions \ref{assMono}, \ref{assErr}, and \ref{assStep} hold, and recall
	$\alpha_k$ computed on line~\ref{linealpha} of
	Algorithm~\ref{AlgfullyAsync}. For all $1 \leq k < \bar{K}$,
	it holds that $\alpha_k\geq\ualpha \triangleq {\ubeta\xi_2}/{\xi_1} > 0$, where
	% \begin{eqnarray}\label{defualpha}
	% 	\underline{\alpha} =\frac{\ubeta\xi_2}{\xi_1},
	% \end{eqnarray}
	$\xi_1$ and $\xi_2$ are as defined in~\eqref{defxi1} and~\eqref{defxi2}. 
	% Furthermore $\underline{\alpha}>0$.
\end{lemma}
\begin{proof}
		By Lemma \ref{propFejer}, $\alpha_k$ defined on line \ref{linealpha}
		of Algorithm~\ref{AlgfullyAsync} may be expressed as
	\begin{align}\label{eqAlphaHyper}
	\alpha_k = 		\frac{\beta_k\varphi_k(z^k,\bw^k)}{\|\nabla\varphi_k\|^2}.
	\end{align}
By Lemma~\ref{lemBoundedGrad}, 
	%\begin{eqnarray*}
	$
		\|\nabla\varphi_k\|^2
		\leq
		\xi_1\sum_{i=1}^{n}
		\| z^k - x_i^k\|^2,
	$
	%\end{eqnarray*}
	where $\xi_1$ is defined in (\ref{defxi1}). 
Furthermore, Lemma \ref{lemLowerBoundPhi} implies that 
%\begin{eqnarray*}
$
\varphi_k(z^k,\bw^k)\geq 
	\xi_2\sum_{i=1}^{n}\| z^k-x^k_i\|^2,
$
%\end{eqnarray*}
	where $\xi_2$ is defined in (\ref{defxi2}). Combining these two
	inequalities with \eqref{eqAlphaHyper} and $\beta_k \geq \ubeta$ yields
	$\alpha_k \geq \ubeta \xi_1 / \xi_2 = \ualpha$.
	% \begin{eqnarray}\nonumber
	% 	\frac{\beta_k\varphi_k(z^k,\bw^k)}{\|\nabla\varphi_k\|^2}
	% 	\geq
	% 	\underline{\alpha}
	% 	=
	% 	\frac{\ubeta\xi_2}{\xi_1}.
	% \end{eqnarray}
	Using \eqref{defxi1} and \eqref{defxi2}, $\xi_1$ and $\xi_2$ are positive
	and finite by assumptions \ref{assErr} and \ref{assStep}.  Since $\beta_k
	> \ubeta > 0$, we conclude that $\ualpha > 0$.
\end{proof}
%!TEX root = convratesv3.tex

% In this section we derive three additional technical lemmas necessary to
% obtain convergence rates for Algorithm \ref{AlgfullyAsync}.

The next lemma is the key to proving $\bigO(1/k)$ function value convergence
 rate and linear convergence rate under strong mononotonicity and
 cocoercivity. Essentially, it shows that several key quanitities of Algorithm
 \ref{AlgfullyAsync} are square-summable, within known bounds.

\begin{lemma}\label{boundSum}\label{lemSum}
Suppose assumptions \ref{assMono}, \ref{assErr}, and \ref{assStep} hold. 
If $\bar{K}=\infty$, then $\varphi_k(p^k)\to 0$ and $\nabla\varphi_k\to 0$. 
Furthermore, for all $1\leq k<\bar{K}$, 
	\begin{align}
		\sum_{t=1}^{k}
	\|z^{t+1}-z^t\|^2
	&\leq \gamma^{-1}\tau\|p^1 - p^*\|^2,
	\label{zsums}\\
	\sum_{t=1}^{k}\sum_{i=1}^{n-1}
	\|w_i^{t+1}-w_i^t\|^2
	&\leq \tau\|p^1 - p^*\|^2,
	\label{wsum} 
\end{align}
\vspace{-4ex}
\begin{align} 
\sum_{i=1}^n\| z^k-x_i^k\|^2
&\leq \frac{\xi_1}{\ubeta^2 \xi_2^2}\|p^{k+1}-p^k\|^2,
	%\,\,\text{ and }\,\,
	&
	\sum_{t=1}^{k}\sum_{i=1}^n \| z^t-x_i^t\|^2
	&\leq \frac{\tau\xi_1}{\ubeta^2 \xi_2^2}\|p^1 - p^*\|^2,
	\label{eqLast1}
	\\
\sum_{i=1}^n\|w_i^k - y_i^k\|^2 
	&\leq 
        E_1\|p^{k+1} - p^k\|^2,
	&
	%\,\,\text{ and }\,\,
	\label{eqLast2}
	\sum_{t=1}^k\sum_{i=1}^n \|w_i^t - y_i^t\|^2 
	&\leq 
		\tau E_1\|p^1 - p^*\|^2, %\label{lemSumeqOld}
	\end{align}
%and
\vspace{-4ex}
\begin{align}
\label{eqLast0}
\sum_{t=1}^k\sum_{i\in\Iforw}\|w_i^t - T_i z^t\|^2 
\leq
\tau E_1\|p^{1} - p^*\|^2,
\end{align}
where 
\begin{align}\label{defE1}
E_1
=
2(1-\sigma)^{-1}\urho^{-1}
(1+\xi_2^{-1}\bar{L}(1+\urho \bar{L}))
\frac{\xi_1}{\ubeta^2 \xi_2},
\end{align}
$\tau$ is as defined in \eqref{eqSums},
and $\xi_1$ and $\xi_2$ are as defined in \eqref{defxi1} and \eqref{defxi2}. 
\end{lemma}
\begin{proof}
    Fix any $1\leq k<\bar{K}$.
	First, from~\eqref{eqSums} in Lemma~\ref{propFejer}, we have
	$
	\sum_{t=1}^k\|p^{t+1}-p^t\|^2\leq \tau\|p^1-p^*\|^2.
	$
	Since
	$
	\|p^{k+1}-p^k\|^2 = \gamma \|z^{k+1}-z^k\|^2+\sum_{i=1}^{n-1}\|w_i^{k+1}-w_i^k\|^2,
    $ 
    inequalities \eqref{zsums} and~\eqref{wsum} follow immediately.
Next, Lemma \ref{propFejer} also implies that
%\begin{align*}
$
p^{k+1} - p^k = -
%\frac
\big(
{\beta_k\varphi_k(p^k)}/{\|\nabla\varphi_k\|^2}
\big)
\nabla\varphi_k,
$
%\end{align*}
and therefore that
\begin{align}\label{eqstepform}
\|p^{k+1}-p^k\|
=
\frac{\beta_k\varphi_k(p^k)}{\|\nabla\varphi_k\|}.
\end{align}
As argued in Lemma~\ref{lemStepBounded}, lemmas~\ref{lemBoundedGrad}
and~\ref{lemLowerBoundPhi} imply that
\begin{align}
\frac{\varphi_k(p^k)}{\|\nabla\varphi_k\|^2}\geq \frac{\xi_2}{\xi_1}
\quad\implies\quad
\|\nabla\varphi_k\|^2\leq \frac{\xi_1}{\xi_2}\varphi_k(p^k).\label{eqRec}
\end{align}
Since $\varphi_k(p^k)\geq 0$ by Lemma \ref{LemGradAffine},
substituting~\eqref{eqRec} into~\eqref{eqstepform} and using the lower bound
on $\beta_k$ yields
\begin{align*}
\varphi_k(p^k) &= \beta_k^{-1}\|\nabla\varphi_k\|\|p^{k+1}-p^k\|
% \\
% &
\leq 
\ubeta^{-1}
\left(\frac{\xi_1}{\xi_2}\right)^{\frac{1}{2}}\!\!\!\sqrt{\varphi_k(p^k)}\,
   \|p^{k+1}-p^k\|,
\end{align*}
which in turn leads to
\begin{align}\label{eqBigs}
\sqrt{\varphi_k(p^k)} &\leq 
\ubeta^{-1} \left(\frac{\xi_1}{\xi_2}\right)^{\frac{1}{2}}\|p^{k+1}-p^k\|
\quad\implies\quad
\varphi_k(p^k) \leq \frac{\xi_1}{\xi_2 \ubeta^2}\|p^{k+1}-p^k\|^2.
\end{align}
Therefore, 
\begin{align}
\sum_{t=1}^k\varphi_t(p^t)\leq \frac{\xi_1}{\xi_2 \ubeta^2}\sum_{t=1}^k\|p^{t+1}-p^t\|^2
\leq
\frac{\tau\xi_1}{\ubeta^2 \xi_2} \|p^1 - p^*\|^2\label{eqMe}.
\end{align}
If $\bar{K}=\infty$, \eqref{eqMe} implies that $\varphi_k(p^k)\to 0$, which in
conjunction with \eqref{eqRec} implies that $\nabla\varphi_k\to 0$. Combining
\eqref{eqBigs} with Lemma \ref{lemLowerBoundPhi} yields the first part of
\eqref{eqLast1}.
%\begin{align*}
%\sum_{i=1}^n
%\|z^k - x_i^k\|^2
%\leq 
%\frac{\xi_1}{\xi_2^2}\|p^{k+1}-p^k\|^2.
%\end{align*}
Applying~\eqref{eqSums} from Lemma \ref{propFejer} then yields the second part
of \eqref{eqLast1}.
%\begin{align*}
%\sum_{k=1}^{K}
%\sum_{i=1}^n
%\|z^k - x_i^k\|^2
%\leq 
%\frac{\xi_1}{\xi_2^2}\|p^1 - p^*\|^2.
%\end{align*}
	
Finally, we establish \eqref{eqLast2} and \eqref{eqLast0}.
Inequality \eqref{eqAltLB} from Lemma~\ref{lemLowerBoundPhi} implies that
\begin{align}
(1-\sigma)\urho\sum_{i\in\Iback} \|w_i^k - y_i^k\|^2
+
\urho\sum_{i\in\Iforw}
\|w_i^k - T_i z^k\|^2
&\leq 
\varphi_k(p^k)+\bar{L}\sum_{i\in\Iforw}\|x_i^k - z^k\|^2
\nonumber\\\label{eqHi}
&
\leq 
(1+\xi_2^{-1} \bar{L})\frac{\xi_1}{\ubeta^2 \xi_2}\|p^{k+1}-p^k\|^2
\end{align} 
where in the second inequality we have used \eqref{eqBigs} and
\eqref{eqLast1}. Together with~\eqref{eqSums} from Lemma~\ref{propFejer},
\eqref{eqHi}~implies \eqref{eqLast0} . Furthermore, for any $i\in\Iforw$, we may write
%\begin{align*}
$
w_i^k - y_i^k
=
w_i^k - T_iz^k + T_i z^k - y_i^k,
$
% \end{align*}
from which Lemma~\ref{lemBasic} can be used to obtain
\begin{align} \label{rewritewTz}
\|w_i^k - T_i z^k\|^2\geq \frac{1}{2}\|w_i^k - y_i^k\|^2 - \|T_i z^k - y_i^k\|^2.
\end{align}
Substituting~\eqref{rewritewTz} into the left hand side of~\eqref{eqHi} yields
\begin{align}
&(1-\sigma)\urho\sum_{i\in\Iback} \|w_i^k - y_i^k\|^2
+
\frac{1}{2}\urho\sum_{i\in\Iforw}
\|w_i^k - y_i^k\|^2
\nonumber\\
&\qquad\qquad\leq 
(1+\xi_2^{-1} \bar{L})\frac{\xi_1}{\ubeta^2 \xi_2}\|p^{k+1}-p^k\|^2
+
\urho \sum_{i\in\Iforw}\|T_i z^k - T_i x_i^k\|^2\nonumber
\\\nonumber
&\qquad\qquad\leq 
(1+\xi_2^{-1} \bar{L})\frac{\xi_1}{\ubeta^2 \xi_2}\|p^{k+1}-p^k\|^2
+
\urho \bar{L}^2\sum_{i\in\Iforw}\|z^k - x_i^k\|^2
\\\nonumber
&\qquad\qquad\leq 
(1+\xi_2^{-1}\bar{L}(1+\urho \bar{L}))\frac{\xi_1}{\ubeta^2 \xi_2}\|p^{k+1}-p^k\|^2
\end{align}
where first inequality follows by substituting $y_k^k = T_i x_i^k$ for
$i\in\Iforw$, the second inequality uses the Lipschitz continuity of $T_i$ for
$i\in\Iforw$, and the final inequality uses \eqref{eqLast1}. Because $0 \leq
\sigma < 1$, we may replace the coefficients in front of the two left-hand
sums by the $(1-\sigma)/2$, which can only be smaller, and obtain
\begin{align}\label{eqLast200}
\sum_{i=1}^n
\|w_i^k - y_i^k\|^2
\leq 
2(1-\sigma)^{-1}\urho^{-1}
(1+\xi_2^{-1}\bar{L}(1+\urho \bar{L}))
\frac{\xi_1}{\xi_2}
\|p^{k+1}-p^k\|^2
\end{align}
which yields the first part of~\eqref{eqLast2}. The second part
follows by applying \eqref{eqSums} to \eqref{eqLast200}.
\end{proof}

\noindent The final technical lemma shows that the sequences $\{x_i^k\}$ and $\{y_i^k\}$
are bounded for $i=1,\ldots,n$, and computes specific bounds on their norms.
\begin{lemma}
Suppose assumptions \ref{assMono}, \ref{assErr}, and \ref{assStep} hold. The sequences $\{x_i^k\}$ and $\{y_i^k\}$ are bounded for $i=1,\ldots,n$. 
In particular, for all  $i=1,\ldots,n$, and $1\leq k\leq\bar{K}$,
\begin{align}
\label{eqUpperX}\|x_i^k\|^2
&\leq 
\min_{p^*\in\calS}
\left\{
2\left(\frac{\xi_1}{\xi_2^2} + 
\gamma^{-1}\right)\|p^1-p^*\|^2+2\gamma^{-1}\|p^*\|^2
\right\}
\triangleq B_x^2
\\
\label{eqUpperY}
\|y_i^k\|^2
&\leq
\min_{p^*\in\calS}
\left\{
2(n+E_1)\|p^1 - p^*\|^2 + 2n\|p^*\|^2
\right\}
\qquad\quad\,
\triangleq B_y^2.
\end{align}\label{lemBoundXY}
\end{lemma}
\vspace{-5ex}
\begin{proof}
Assumption \ref{assMono} asserts $\calS$ is nonempty, so let $p^*\in\calS$ and
fix any $1 \leq k < \bar{K}$. From Lemma~\ref{lemBasic},
\begin{align}
\|p^k\|^2
&=
\|p^k-p^*+p^*\|^2
%\nonumber\\
%&
\leq
2\|p^k-p^*\|^2+2\|p^*\|^2
%\nonumber\\\nonumber 
%&
\leq 
2\|p^1-p^*\|^2+2\|p^*\|^2,
\end{align}
where the final inequality uses \eqref{eqFejer}.
It immediately follows that
\begin{align}\label{eqzUB}
\|z^k\|^2&\leq \gamma^{-1}\|p^k\|^2 \leq 2\gamma^{-1}\|p^1-p^*\|^2+2\gamma^{-1}\|p^*\|^2
\\
% \end{align}
% and 
% \begin{align}
\label{eqWup} 
\sum_{i=1}^{n-1}\|w_i^k\|^2 &\leq \|p^k\|^2\leq 2\|p^1-p^*\|^2+2\|p^*\|^2.
\end{align}
Furthermore, using the definition of $w_n^k$ and Lemma~\ref{lemBasic}, we also have
\begin{align}\label{eqwncase}
\|w_n^k\|^2 = 
\norm{\sum_{i=1}^{n-1}w_i^k}^2
\leq 
(n-1) \sum_{i=1}^{n-1}\norm{w_i^k}^2
\leq 
2n\|p^1-p^*\|^2+2n\|p^*\|^2.
\end{align}
Therefore, for all $i=1,\ldots, n$, we have
\begin{align}
\|x_i^k\|^2
&\leq 
2\|z^k-x_i^k\|^2+2\|z^k\|^2
%\nonumber\\
\label{eq2min1}
%&
\leq
2\left(\frac{\xi_1}{\ubeta^2 \xi_2^2} + 
\gamma^{-1}\right)\|p^1-p^*\|^2+2\gamma^{-1}\|p^*\|^2,
\end{align}
where second inequality uses \eqref{eqzUB} and \eqref{eqLast1}.
Finally, for $i=1,\ldots,n$
\begin{align}
\|y_i^k\|^2
&\leq
2\|y_i^k - w_i^k\|^2 + 2\|w_i^k\|^2
%\nonumber\\
\label{eq2min2}
%&
\leq 
2(n+E_1)\|p^1 - p^*\|^2 + 2n\|p^*\|^2,
\end{align}
where  the second inequality uses \eqref{eqLast2} and
\eqref{eqWup}-\eqref{eqwncase}. Note that the factor of $n$ is only necessary
when $i=n$, but for simplicity we will just a single bound for all $i$.

Finally, the set $\calS$ is closed and convex from
Proposition~\ref{prop:SClosedConv}, and the expressions \eqref{eq2min1} and
\eqref{eq2min2} are continuous functions of $p^*$. Thus, we may minimize these
bounds over $\calS$, yielding
\eqref{eqUpperX}-\eqref{eqUpperY}.
\end{proof}

%\input{funcVals2Terms}
%!TEX root = convratesv3.tex
\section{Function Value Convergence Rate}\label{secFunc}
\subsection{Assumptions}
We now consider the optimization problem given in \eqref{ProbOpt0}, which we repeat here:
\begin{eqnarray}\label{ProbOpt}
F^*=\min_{z\in\calH} F(z) =\min_{z\in\calH}  \sum_{i=1}^{n}f_i(z).
\end{eqnarray}
\begin{assumption}
	\label{assOPT}
Problem \eqref{ProbOpt} conforms to the following:
	\begin{enumerate}
		\item $\calH$ is a real Hilbert space. 
		\item Each $f_i:\calH\to(-\infty,+\infty]$ is convex, closed, and proper. 
		\item There exists some subset $\Iforw\subseteq\{1,\ldots,n\}$ s.t.
		for all $i\in\Iforw$, $f$ is Fr\'echet differentiable everywhere and
		$\nabla f_{i}$ is $L_i$-Lipschitz continuous.
		\item For $i\in\Iback\triangleq \{1,\ldots,n\} \backslash \Iforw$,
		$\prox_{\rho f_i}$ can be computed 
		to within the error tolerance specified in Assumption \ref{assErr} 
		(however, these functions are not
		precluded from also having Lipschitz continuous gradients).
		\item Letting $T_i = \partial f_i$ for $i=1,\ldots,n$, the solution
		set $\calS$ in \eqref{defCompExtSol} is nonempty.
	\end{enumerate}
\end{assumption}
\noindent When $\lvert\Iback\rvert \geq 1$, it is possible for~\eqref{ProbOpt} to
have a finite solution, yet for $\calS$ to be empty.  For
constraint-qualification-like conditions precluding such pathological
situations, see for example~\cite[Corollary 16.50]{bauschke2011convex}.

\begin{lemma}\label{LemOptStuff}
If Assumption \ref{assOPT} holds, then Assumption \ref{AssMonoProb} holds for \eqref{probCompMonoMany} with $T_i = \partial f_i$ for $i=1,\ldots,n$. If $(z,\bw)\in\calS$, then $z$ is a solution to \eqref{ProbOpt}.  The solution value $F^*$ is finite. 
\end{lemma} 
\begin{proof}
	Closed, convex, and proper functions have maximal monotone
	subdifferentials \cite[Theorem 21.2]{bauschke2011convex}. That $z$ is a
	solution to \eqref{ProbOpt} for any $(z,\bw)\in\calS$ is a consequence of
	Fermat's rule \cite[Proposition 27.1]{bauschke2011convex} and the
	elementary fact that $\sum_i\partial f_i(x)\subseteq\partial(\sum_i
	f_i)(x)$. Since $\calS$ is nonempty, a solution to \eqref{ProbOpt} exists.
	For any $(z,\bw)\in\calS$, the function $f_i$ must be subdifferentiable
	and hence finite at $z$ for all $i=1,\ldots,n$, so $F^* = \sum_{i=1}^n
	f_i(z)$ must be finite.
\end{proof}

In the following analysis, we consider two types of convergence rates. First
we will establish a rate of the form
\begin{align*}
\sum_{i=1}^n f_i(\bar{x}_i^k) - F^* = \bigO(1/k)
\quad\text{ and }\quad
\|\bar{x}_i^k - \bar{x}_j^k\| = \bigO(1/k)\quad\forall i,j=1,\ldots,n, 
\end{align*}
where $\bar{x}_i^k$ is an appropriate averaging of the sequence $\{x_i^k\}$. With an additional Lipschitz continuity assumption, we can derive a more direct rate of the form
\begin{align}\label{eqBetterConv}
\sum_{i=1}^n f_i(\bar{x}_j^k) - F^* = \bigO(1/k)
\end{align}
For an appropriate index $j$.
This additional assumption is as follows:

\begin{assumption}
\label{assLip}
If $|\Iback|>1$, there exists $\Ilip\subseteq \Iback$ such that
$|\Ilip|\geq |\Iback|-1$, 
and for all $i\in\Ilip$ the function $f_i$ is $M_i$-Lipschitz continuous on the ball: $\{x:\|x\|\leq B_x\}$ where $B_x$ is defined in \eqref{eqUpperX}. 
\end{assumption}

By virtue of Assumption \ref{assLip}, note that for all $i\in\Ilip$, $\{x:\|x\|\leq B_x\}\subseteq\dom(f_i)$. 

We point out that it is not surprising that Assumption \ref{assLip} is
required to derive convergence rates of the form \eqref{eqBetterConv}: suppose
the first two functions $f_1$ and $f_2$ are the respective indicator functions
of closed nonempty convex sets $\calC_1$ and $\calC_2$, that is, $f_i(x)=0$ if
$x\in\calC_i$ and otherwise $f_i(x)=+\infty$. Since neither function is
differentiable, $\{1,2\}\subseteq \Iback$. Since neither function is Lipschitz
continuous, $\{1,2\}\notin\Ilip$, unless $B_x\subseteq
\calC_1\cap\calC_2$.  This last situation is of little interest, since the
iterates remain inside $\calC_1 \cap \calC_2$ for all iterations and thus
the constraints encoded by $f_1$ and $f_2$ may be disregarded.

Instead, suppose $B_x \not\subseteq \calC_1\cap\calC_2$ and thus $\{1,2\}\notin
\Ilip$. Then $|\Ilip|\leq |\Iback|-2$, which violates Assumption \ref{assLip}.
Suppose anyway that we could establish a convergence rate of $\bigO(1/k)$ (or
similar) for some sequence of points $\tilde{x}^k$ generated by the algorithm.
But this would imply that $\tilde{x}^k\in\calC_1\cap\calC_2$ for all $k$,
meaning that we would have to be able to solve an arbitrary two-set convex feasibility
problem in a single iteration.  Of course, it is easy to construct a
counterexample in which Algorithm \ref{AlgfullyAsync} does not find a point in
the intersection of two closed convex sets within a finite number of
iterations.

 So, to derive a convergence rate of the form \eqref{eqBetterConv} we allow
 for only one function to be non-Lipschitz and thus able to encode a
 ``hard''  constraint. Any other nonsmooth functions present in the problem
 must be Lipschitz continuous on the bounded set $B_x$ which contains all
 points  potentially encountered by the algorithm.

 \subsection{Main Result}

%In the following theorem we prove an $O(1/k)$ ergodic function covergence rate for projective splitting applied to \eqref{ProbOpt}. The averages are evaluated w.r.t. $x_j^k$ where $j$ is the unique element of $\Iback\backslash\Ilip$, unless $\Iback$ is empty in which case $j$ can be any element of $\Iforw$. 

\begin{theorem}
Suppose assumptions \ref{assErr}, \ref{assStep}, and \ref{assOPT} hold. 
For $1\leq k\leq \bar{K}$, let
\begin{eqnarray*}
(\forall i=1,\ldots,n)\quad \bar{x}_{i}^k = \frac{\sum_{t=1}^k \alpha_t x_{i}^t}{\sum_{t=1}^k\alpha_t},
\end{eqnarray*}
where $\alpha_t$ is calculated on line~\ref{linealpha} of Algorithm~\ref{AlgfullyAsync}. Fixing any $p^*\in\calS$, one has
\begin{enumerate}
	\item 
	If $\bar{K}<\infty$,
	\begin{align}\label{eqFint}
	(\forall j=1,\ldots,n)\quad 
	\sum_{i=1}^{n} f_i(x_j^{\bar{K}}) = \sum_{i=1}^{n} f_i(z^{\bar{K}+1})= F^*,
	\end{align}
	\item For all $1\leq k<\bar{K}$ 
\begin{eqnarray}\label{eqMultiFuncResult0}
\sum_{i=1}^{n} f_i(\bar{x}^k_i) - F^*
\leq
\frac{E_2\|p^1 - p^*\|^2 + E_3\|p^1 - p^*\|}{k}
\end{eqnarray}
where 
\begin{align}\label{defE2new}
E_2 &= 
\frac{\xi_1}{2\ubeta\xi_2}\left(
1 + (3+2E_1)\tau + \orho_n\tau\!\left( 2 + \gamma E_1 +
\frac{\gamma\delta\xi_1}{\ubeta^2\xi_2^2}\right)
\right)
\\
E_3 &=2\sqrt{n}\|p^*\|\frac{\xi_1}{\ubeta\xi_2}.
\label{defE3}
\end{align}
Furthermore, 
\begin{align}\label{defergup}
(\forall i,l=1,\ldots,n)\qquad\|\bar{x}_i^k-  \bar{x}_l^k\|
\leq
\frac{4 \|p^1-p^*\|}{\ualpha k} =
\frac{4\xi_1\|p^1-p^*\|}{\ubeta \xi_2 k}.
\end{align}
%and $E_2\geq 0$ and $E_3\geq 0$ are given in \eqref{defE2new} and \eqref{defE3}.
\item 
Additionally suppose Assumption \ref{assLip} holds. If $\Iback=1$, let $j$ be
the unique element in $\Iback$. If $|\Iback|>1$, let $j$ be the unique element
of $\Iback\backslash\Ilip$. If $|\Iback|=0$ then choose $j$ to be any index in
$\Iforw$. Then, for all $1\leq k<\bar{K}$,
\begin{eqnarray}\label{eqMultiFuncResult}
\sum_{i=1}^{n} f_i(\bar{x}^k_j) - F^*
\leq
\frac{E_2\|p^1-p^*\|^2 + E_4\|p^1 - p^*\|}{k}
\end{eqnarray}
where 
\begin{align}\label{defE4}
E_4 = E_3 + 
\frac{4\xi_1}{\ubeta\xi_2}
\left(
\sum_{i\in\Ilip}
M_i
+
nB_y(1
+
2\bar{L} B_x)
\right).
\end{align}
\end{enumerate}
\end{theorem}

\begin{proof}
	In the case that $\bar{K}<\infty$, it was established in \cite[Lemma
	5]{johnstone2018projective} that $x_1^{\bar{K}} = \cdots = x_n^{\bar{K}} =
	z^{\bar{K}+1}$ is a solution to \eqref{ProbOpt}, which establishes
	\eqref{eqFint}. We now address points 2 and 3.
	
	\noindent 

\vspace{0.2cm}
\siamversion{\pagebreak[2]}
\underline{Part 1: An upper bound for $\sum_{i=1}^{n}f_i(\bar{x}_{i}^k)-F^*$}
\vspace{0.2cm}

We begin by establishing the upper bound on
$\sum_{i=1}^{n}f_i(\bar{x}_{i}^k)-F^*$ in \eqref{eqMultiFuncResult0}. Since
$f_i$ is convex,
\begin{align}
\sum_{i=1}^n f_i(\bar{x}_i^k)-F^*
&=
\sum_{i=1}^n 
f_i\left(
\frac{\sum_{t=1}^k\alpha_t x_i^t}{\sum_{t=1}^k\alpha_t}
\right)-
F^*
% \\
% &
\leq
\frac{\sum_{t=1}^k \alpha_t\left(\sum_{i=1}^n  f_i(x_i^t)-F^*\right)}
{
	\sum_{t=1}^k \alpha_t
}.
\label{eqAv}
\end{align}
Since Lemma \ref{lb_ratio} implies that $\alpha_k\geq\ualpha$, we will aim to
 show that  $\sum_{t=1}^k \alpha_t\left(\sum_{i=1}^n f_i(x_i^t)-F^*\right)$ is
 bounded for all $1\leq k<\bar{K}$ (recall that $\bar{K}$ may be $+\infty$).

The projection updates on lines \ref{eqAlgproj1}-\ref{eqAlgproj2} of Algorithm~\ref{AlgfullyAsync} mean that, for all $1\leq k<\bar{K}$,
\begin{align}
z^{k+1} &= z^k - \gamma^{-1}\alpha_k\sum_{i=1}^n y_i^k
\label{zUpdate}\\\label{wUpdate}
w_i^{k+1} &= w_i^k - \alpha_k(x_i^k - x_n^k),\quad i=1,\ldots,n-1.
\end{align}
Take any $p^* = (x,\bw)\in\calS$.  By Lemma \ref{LemOptStuff}, $F^* =
\sum_{i=1}^n f_i(x)$. Then
\begin{eqnarray}
\sum_{i=1}^{n}f_i(x_i^k) -F^*
&=&
\sum_{i=1}^{n}f_i(x_i^k) -\sum_{i=1}^{n} f_i(x)
\nonumber \\
&\overset{(a)}{\leq}&
\sum_{i=1}^{n}\langle y_i^k,x_i^k - x\rangle
\nonumber\\
&=&
\left\langle \sum_{i=1}^{n} y_i^k,x_{n}^k - x\right\rangle
+
\sum_{i=1}^{n-1}\langle y_i^k,x_i^k - x_{n}^k\rangle
\nonumber\\\label{multi}
&\overset{(b)}{=}&
\underbrace{
	\frac{\gamma}{\alpha_k}\left\langle z^k-z^{k+1},x^k_{n} - x\right\rangle}_{\triangleq A_1^k}
+
\underbrace{\sum_{i=1}^{n-1}\langle y_i^k,x_i^k - x^k_{n}\rangle}_{\triangleq A_2^k}.
\end{eqnarray}
In the above, (a) uses that $y_i^k\in\partial f_i(x_i^k)$ and (b) uses
(\ref{zUpdate}).  We now show that $\alpha_k A_1^k$ and $\alpha_k A_2^k$ both
have a finite sum over $k$.

\vspace{0.2cm}
\noindent\underline{ $\alpha_k A_1^k$ is Summable}
\vspace{0.2cm}

\noindent 
If $n\in\Iback$, $A_1^k$ can be simplified as follows:
\begin{eqnarray}
A_1^k
&\overset{(a)}{=}&
\frac{\gamma}{\alpha_k}
\left\langle z^k-z^{k+1},z^k - x+\rho_{n}^k(w^k_{n}-y_n^k)
+e_n^k
\right\rangle
\nonumber\\\nonumber
&=&
\frac{\gamma}{\alpha_k}\left\langle z^k-z^{k+1},z^k - x\right\rangle
+
\frac{\gamma\rho_{n}^k}{\alpha_k}
\left\langle z^k-z^{k+1},
w_{n}^k-y_n^k\right\rangle
\nonumber\\ & & \quad
+ \;
\frac{\gamma\rho_{n}^k}{\alpha_k}
\left\langle z^k-z^{k+1},
e_n^k\right\rangle
\\
&\overset{(b)}{\leq}&
\frac{\gamma}{\alpha_k}\left\langle z^k-z^{k+1},z^k - x\right\rangle
+
\frac{\gamma\rho_{n}^k}{\alpha_k}
\left\langle z^k-z^{k+1},
w_{n}^k-y_n^k\right\rangle
\nonumber\\
&&
\qquad + \;
\frac{\gamma\rho_{n}^k}{2\alpha_k}
\left(
\|z^k-z^{k+1}\|^2
+
\|
e_n^k\|^2
\right)
\nonumber\\
% \siamversion{\pagebreak[2]}
&\overset{(c)}{\leq}& 
\frac{\gamma}{\alpha_k}\left\langle z^k-z^{k+1},z^k - x\right\rangle
+
\frac{\gamma\rho_{n}^k}{\alpha_k}
\left\langle z^k-z^{k+1},
w_{n}^k-y_n^k\right\rangle
\nonumber\\\label{agh}
&&
\qquad 
+ \;
\frac{\gamma\orho_{n}}{2\alpha_k}
\left(
\|z^k-z^{k+1}\|^2
+
\delta\|z^k - x_n^k\|^2
\right)
\end{eqnarray}
where (a) uses the proximal update on line \ref{LinebackwardUpdate} of the
algorithm for the case $i=n$, (b) used Young's inequality, and (c) uses
assumptions \ref{assErr} and \ref{assStep}. On the other hand, if
$n\in\Iforw$, $A_1^k$ may be written as
\begin{align}
A_1^k
=
\frac{\gamma}{\alpha_k}\left\langle z^k-z^{k+1},z^k - x\right\rangle
+
\frac{\gamma\rho_{n}^k}{\alpha_k}
\left\langle z^k-z^{k+1},
w_{n}^k-T_n z^k\right\rangle,
\label{agh2}
\end{align}
where we have instead used the forward step on line \ref{ForwardxUpdate}. Let
$\chi^k = w_n^k - y_n^k$ when $n\in\Iback$ and $\chi^k = w_n^k -T_n z^k$ when
$n\in\Iforw$. Furthermore, let
\begin{equation} \label{defThetak}
\theta_k = \frac{\gamma\orho_{n}}{2\alpha_k}
\left(
\|z^k-z^{k+1}\|^2
+
\delta\|z^k - x_n^k\|^2
\right)
\end{equation}
when $n\in \Iback$ and $\theta_k = 0$ if $n\in\Iforw$. 
Combining \eqref{agh} and \eqref{agh2}, we may write 
\begin{align}
A_1^k
=
\frac{\gamma}{\alpha_k}\left\langle z^k-z^{k+1},z^k - x\right\rangle
+
\frac{\gamma\rho_{n}^k}{\alpha_k}
\left\langle z^k-z^{k+1},
\chi^k\right\rangle
 +\theta_k.
\label{agh3}
\end{align}
Rewriting the first term in~\eqref{agh3} using (\ref{eqCosine}), we obtain
\begin{eqnarray}\label{comb00}
\frac{\gamma}{\alpha_k}\left\langle z^k-z^{k+1},z^k - x\right\rangle
&=&
\frac{\gamma}{2\alpha_k}\|z^k-z^{k+1}\|^2 
     + \frac{\gamma}{2\alpha_k}\left(\|z^k-x\|^2 -\|z^{k+1}-x\|^2\right)
\end{eqnarray}
The second term in (\ref{agh3}) may be upper bounded using Young's inequality, as
follows:
\begin{align}
\nonumber 
\frac{\gamma\rho_{n}^k}{\alpha_k}
\left\langle z^k-z^{k+1},
\chi^k\right\rangle
\leq 
\frac{\gamma\rho_{n}^k}{2\alpha_k}\|z^k - z^{k+1}\|^2
+
\frac{\gamma\rho_{n}^k}{2\alpha_k}
\|\chi^k\|^2
\end{align}
Thus, we obtain
\begin{align}\label{eqpassLong}
\alpha_k A_1^k
\leq
\frac{\gamma}{2}\left(1+\overline{\rho}_{n}\right)\|z^k-z^{k+1}\|^2 + \frac{\gamma}{2}\|z^k-x\|^2 -\frac{\gamma}{2}\|z^{k+1}-x\|^2
+
\frac{\gamma\overline{\rho}_{n}}{2}
\|\chi^k\|^2
+
\alpha_k\theta_k.
\end{align}
Summing \eqref{eqpassLong} over $k$ yields, for $1\leq k<\bar{K}$,
\begin{align} 
\sum_{t=1}^k \alpha_t A_1^t
\leq
\frac{\gamma}{2}\|z^1-x\|^2
+
\frac{\gamma}{2}\left(1+\overline{\rho}_{n}\right)\sum_{t=1}^k \|z^t-z^{t+1}\|^2 
+
\frac{\gamma\overline{\rho}_{n}}{2}
\sum_{t=1}^k
\|\chi^t\|^2
+\sum_{t=1}^k\alpha_t \theta_t
.\label{eqWrk1}
\end{align}  
We now consider the first three terms on the right-hand side of this relation, employing
Lemma~\ref{lemSum}:
\begin{align*}
\frac{\gamma}{2}\|z^1-x\|^2
&\leq 
\frac{1}{2} \smallnorm{p^1 - p^*}^2
&&&& \text{by \eqref{gammanorm}} \\
\frac{\gamma}{2}\left(1+\overline{\rho}_{n}\right)\sum_{t=1}^k \|z^t-z^{t+1}\|^2
&\leq
\frac{1}{2}(1 + \overline{\rho}_n)\tau \smallnorm{p^1 - p^*}^2
&&&& \text{by \eqref{zsums}} \\
\frac{\gamma\overline{\rho}_{n}}{2}
\sum_{t=1}^k
\|\chi^t\|^2
&\leq
\frac{\gamma\orho_{n}}{2} \tau E_1 \smallnorm{p^1 - p^*}^2
&&&& \text{by \eqref{eqLast2} or \eqref{eqLast0}.}
\end{align*}
In the last inequality, we use~\eqref{eqLast2} when $n\in\Iback$
and~\eqref{eqLast0} when $n\in\Iforw$, and $E_1$ is defined in \eqref{defE1}.
We now consider the last term in~\eqref{eqWrk1}.  When $n\in\Iback$, we have
\begin{align*}
\sum_{t=1}^k\alpha_t \theta_t
&=
\frac{\gamma\orho_{n}}{2}
\left(
\sum_{t=1}^k \|z^k-z^{k+1}\|^2
+
\sum_{t=1}^k \delta\|z^k - x_n^k\|^2
\right)
&&&& \text{by~\eqref{defThetak}} \\
&\leq
\frac{\gamma\orho_{n}}{2}
\left(
\gamma^{-1}\tau \smallnorm{p^1 - p^*}^2 
+ 
\delta \frac{\tau\xi_1}{\ubeta^2\xi_2^2}  \smallnorm{p^1 - p^*}^2
\right)
&&&& \text{by~\eqref{zsums} and~\eqref{eqLast1}} \\
&=
\frac{\tau\orho_n}{2}\left(1 + \frac{\gamma\delta\xi_1}{\ubeta^2\xi_2^2} \right) 
    \smallnorm{p^1 - p^*}^2.
\end{align*}
The resulting inequality also holds trivially when $n\in\Iforw$, since its
left-hand side must be zero. Combining all these inequalities, we obtain
\begin{align}\label{eqbig1}
\sum_{t=1}^k \alpha_t A_1^t
&\leq
% \left(
% \frac{3}{2}
% +
% \overline{\rho}_{n}
% +
% \gamma\orho_n
% \left(
% \frac{\delta \xi_1}{\xi_2^2}
% +
% \frac{1}{2}E_1
% \right)
% \right)
\frac{1}{2}
\left(
1 + (1+\orho_n)\tau + \gamma\orho_n\tau E_1 
+ \tau\orho_n\!\left(1 + \frac{\gamma\delta\xi_1}{\ubeta^2\xi_2^2}\right)
\right)
\|p^1-p^*\|^2 \nonumber \\
&=
\frac{1}{2}\left(
1 + \tau + \orho_n\tau\!\left( 2 + \gamma E_1 +
  \frac{\gamma\delta\xi_1}{\ubeta^2\xi_2^2}\right)
\right)
\|p^1-p^*\|^2.
\end{align}
% where $E_1$ is given in \eqref{defE1} and we used the following derivations.
% The first term in \eqref{eqWrk1} was dealt with via \eqref{zbegin}. The second
% term via \eqref{zsums} in Lemma \ref{lemSum}. The third term via
% \eqref{eqLast2} when $n\in\Iback$ and \eqref{eqLast0} when $n\in\Iforw$ in
% Lemma \ref{lemSum}. The final term in \eqref{eqWrk1} was dealt with via
% \eqref{zsums} and \eqref{eqLast1}.

%Now using \eqref{zsums} and \eqref{eqLast2} of Lemma \ref{lemSum},
%telescoping and then using \eqref{zbegin}, \eqref{eqpassLong} implies for all
%$K<\bar{K}$
%\begin{align}
%\label{eqInter}
%\sum_{k=1}^{K}\alpha_k A_1^k
%\leq 
%D_1\|p^1 - p^*\|^2
%\end{align}
%where $D_1$ is computed explicitly in appendix \ref{appenD1}. 

\vspace{0.2cm}
\noindent\underline{$\alpha_k A^k_2$ is Summable}
\vspace{0.2cm}

\noindent
We next perform a similar summability analysis on the second term
in~\eqref{multi}, $A_2^k$. We begin by fixing any $1 \leq k < \bar{K}$ and writing
\begin{align}
\alpha_k A_2^k &= \alpha_k  \sum_{i=1}^{n-1} \langle y_i^k,x_i^k - x_{n}^k\rangle
% \nonumber\\
% &=&
=
\alpha_k \sum_{i=1}^{n-1} \langle w_i^k,x_i^k - x_{n}^k\rangle
+
\alpha_k \sum_{i=1}^{n-1} \langle y_i^k - w_i^k,x_i^k - x_{n}^k\rangle.
\label{comb2}
\end{align}
Now fix any $i = 1,\ldots,n-1$.
Using \eqref{wUpdate}, 
\begin{align}\label{eqKeys}
x_i^k - x_n^k = -\frac{1}{\alpha_k}(w_i^{k+1}-w_i^k).
\end{align}
Substituting this into the first summand in \eqref{comb2} and then
using~\eqref{eqCosine}, we have
\begin{multline}
\alpha_k\langle w_i^k,x_i^k - x_{n}^k\rangle
=
 \langle w_i^k,w_i^k - w_i^{k+1}\rangle
%\nonumber
\\ %\nonumber
=
 \langle w_i^k-w_i^*,w_i^k - w_i^{k+1}\rangle
 +
  \langle w_i^*,w_i^k - w_i^{k+1}\rangle \text{~~~~~~~~~~~~~~~~~~~~~~~~~}
\\\label{equsecos}
=
\frac{1}{2}
\left(
\|w_i^k-w^*_i\|^2
+
\|w_i^k - w_i^{k+1}\|^2
-
\|w_i^{k+1}-w_i^*\|^2
\right) 
+
  \langle w_i^*,w_i^k - w_i^{k+1}\rangle.
\end{multline}
Using \eqref{eqKeys} in the second summand in \eqref{comb2} and applying
Young's inequality yields
\begin{align}
\alpha_k\langle y_i^k - w_i^k,x_i^k - x_{n}^k\rangle
&=
\langle y_i^k - w_i^k,w_i^k - w_i^{k+1}\rangle
% \nonumber\\
% &
\leq \label{eqCombine2}
\frac{1}{2}
\|
y_i^k - w_i^k
\|^2
+
\frac{1}{2}
\|
w_i^k - w_i^{k+1}
\|^2.
\end{align}
Substituting \eqref{equsecos} and \eqref{eqCombine2} into \eqref{comb2} yields
\begin{align*}
\alpha_k A_2^k
&\leq 
\frac{1}{2}
\sum_{i=1}^{n-1} 
\left(
\|w_i^k-w_i^*\|^2
+2
\|w_i^k - w_i^{k+1}\|^2
-
\|w_i^{k+1}-w_i^*\|^2
+
\|
y_i^k - w_i^k
\|^2
\right)
\\
&\qquad
+\;
\sum_{i=1}^{n-1}\langle w_i^*,w_i^k - w_i^{k+1}\rangle
.
\end{align*}
Summing this relation, we obtain that for any $1 \leq k<\bar{K}$,
\begin{align}
\sum_{t=1}^k\alpha_t A_2^t
&\leq
\frac{1}{2}\sum_{i=1}^{n-1} 
\left(
\|w_i^1-w_i^*\|^2
+
2
\sum_{t=1}^k
\|w_i^t - w_i^{t+1}\|^2
+
\sum_{t=1}^k
\|
y_i^t - w_i^t
\|^2
\right)
\nonumber\\
&\qquad + \;
\sum_{i=1}^{n-1}\langle w_i^*,w_i^1 - w_i^{k+1}\rangle
\\\nonumber
&\leq 
\frac{1}{2}\sum_{i=1}^{n-1} 
\left(
\|w_i^1-w_i^*\|^2
+
2
\sum_{t=1}^k
\|w_i^t - w_i^{t+1}\|^2
+
\sum_{t=1}^k
\|
y_i^t - w_i^t
\|^2
\right)
\\\label{eqNeed}
& \qquad + \;
\sum_{i=1}^{n-1}\|w_i^*\|(\|w_i^1-w_i^*\|+\|w_i^{k+1}-w_i^*\|).
\end{align}
By \eqref{wsum} and \eqref{eqLast2} from Lemma \ref{lemSum}, we then obtain
\begin{align*}
\sum_{t=1}^k
\sum_{i=1}^{n-1}
\left(
\|w_i^t - w_i^{t+1}\|^2
+
\sum_{t=1}^k
\|
y_i^t - w_i^t
\|^2
\right)
&\leq 
\tau (1+E_1)\|p^1-p^*\|^2.
\end{align*}
Finally, we may use \eqref{wbegin} to derive
\begin{align*}
\sum_{i=1}^{n-1}\|w_i^*\|(\|w_i^1-w_i^*\|+\|w_i^{k+1}-w_i^*\|)
&\leq 
2\|p^1-p^*\|\sum_{i=1}^{n-1}\|w_i^*\|
\\
&\leq 
2\sqrt{n}\|p^1-p^*\|\|\bw^*\|
% \\
% &
\leq 
2\sqrt{n}\|p^1-p^*\|\|p^*\|,
\end{align*}
where in the second inequality we have used Lemma \ref{lemBasic}. 
Therefore,
\begin{align}\label{eqBiggs}
\sum_{t=1}^k\alpha_t A_2^t
\leq 
\tau (1+E_1)\|p^1-p^*\|^2
+
2\sqrt{n}\|p^*\|\|p^1-p^*\|.
\end{align}

%Returning to \eqref{eqAv} using \eqref{eqInter} and \eqref{eqInter2} in \eqref{multi} we see that
%\begin{align}
%\sum_{i=1}^n f_i(\bar{x}_i^k)-F^*
%&\leq
%\frac{	\sum_{l=1}^k \alpha_l\left(\sum_{i=1}^n  f_i(x_i^l)-F^*\right)}
%{
%	\sum_{l=1}^k \alpha_l
%}	
%\nonumber 
%\\\nonumber 
%&\leq 
%\frac{(D_1+D_2)\|p^1 - p^*\|^2+ D_3\|p^1-p^*\|}
%{
%	\sum_{l=1}^k \alpha_l
%}
%\\
%&\leq
%\label{eqAv2}
%\frac{(D_1+D_2)\|p^1 - p^*\|^2+D_3\|p^1-p^*\|}
%{
%	\underline{\alpha}k
%}.
%\end{align}

\noindent Recalling \eqref{eqAv} and using \eqref{multi} and the fact that $\alpha_t\geq
\ualpha$ by Lemma \ref{LemEveryOtherIter}, we obtain, for any $1\leq k<\bar{K}$,
\begin{align}
\sum_{i=1}^n f_i(\bar{x}_i^k)-F^*
&\leq
\frac{	\sum_{t=1}^k \alpha_t\left(\sum_{i=1}^n  f_i(x_i^t)-F^*\right)}
{
	\sum_{t=1}^k \alpha_t
}
\nonumber\\
&\leq 
\frac{\sum_{t=1}^k\alpha_t(A_1^t+A_2^t)}
{
	\sum_{t=1}^k \alpha_t
}
\nonumber\\\label{eqFFI}
&\leq 
\frac{\sum_{t=1}^k\alpha_t(A_1^t+A_2^t)}
{
\ualpha k
}
=
\frac{\xi_1}{\ubeta \xi_2}\cdot
\frac{\sum_{t=1}^k\alpha_t(A_1^t+A_2^t)}{k}.
\end{align}
Using \eqref{eqbig1} and \eqref{eqBiggs} in \eqref{eqFFI} establishes \eqref{eqMultiFuncResult0} with $E_2$ and $E_3$ defined in \eqref{defE2new} and \eqref{defE3}.  

%For a more explicit computation of $E_2$ and $E_3$, see Appendix \ref{appendEcalc}. XX

\vspace{0.2cm}
	\noindent 
\underline{Part 2: Convergence behavior of $\bar{x}_i^k - \bar{x}_l^k$}
\vspace{0.2cm}

\noindent	
Recall from \eqref{wUpdate} that for all $i=1,\ldots,n-1$,
\begin{align*}
w_i^{t+1} -  w_i^t = \alpha_t(x_i^t - x_n^t).
\end{align*}
Summing this equation over $t=1,\ldots,k$ and dividing by $\sum_{t=1}^k \alpha_t$ yields
\begin{align}\label{eq4n}
(\forall\;i=1,\ldots,n-1) \qquad
\bar{x}_n^k - \bar{x}_i^k = \frac{w_i^1 -  w_i^{k+1}}{\sum_{t=1}^k\alpha_t}.
\end{align}
Therefore,  we obtain for all $i,l \in \{1,\ldots,n-1\}$ that
\begin{align*}
\bar{x}_l^k - \bar{x}_i^k 
&= 
\frac{w_i^1 - w_l^1 - (w_i^{k+1}-w_l^{k+1})}{\sum_{t=1}^k\alpha_t}
\\
&=
\frac{w_i^1-w_i^* - w_l^1+w_l^* - (w_i^{k+1}-w_i^* -w_l^{k+1}+w_l^*)}
     {\sum_{t=1}^k\alpha_t}.
\end{align*}
Using \eqref{wbegin} and the triangle inequality, we therefore have for all
$i,l \in \{1,\ldots,n-1\}$ that
\begin{align}\label{defergup1}
\|\bar{x}_l^k - \bar{x}_i^k\| \leq 
\frac{\|w_l^1-w_l^*\| +\|w_i^1-w_i^*\|+ \|w_l^{k+1}-w_l^*\|+\|w_i^{k+1}-w_i^*\|}
     {\sum_{t=1}^k\alpha_t}
\leq 
\frac{4\|p^1-p^*\|}{\underline{\alpha}k},
\end{align}
where  the second inequality uses lemmas \ref{propFejer} and \ref{lemStepBounded}. 
Finally, we can also use \eqref{eq4n} to obtain, for any $i=1,\ldots,n-1$, that
\begin{align}\label{defergup2}
\norm{\bar{x}_n^k - \bar{x}_i^k}
&= \frac{\smallnorm{w_i^1 - w_i^* - (w_i^k - w_i^*)}}{\sum_{t=1}^k\alpha_t}
\leq \frac{\smallnorm{w_i^1 - w_i^*} + \smallnorm{w_i^k - w_i^*)}}{\sum_{t=1}^k\alpha_t}
\leq 
\frac{
	2\norm{p^1 - p^*}
}{\ualpha k},
\end{align}
where the second inequality again uses lemmas \ref{propFejer} and
\ref{lemStepBounded}. Together, \eqref{defergup1}-\eqref{defergup2} and the
definition of $\ualpha$ imply \eqref{defergup}.

\vspace{0.2cm}
\noindent 
\underline{Part 3: Convergence rates at a single ``splitting" point}
\vspace{0.2cm}

\noindent 
We now prove \eqref{eqMultiFuncResult} under Assumption \ref{assLip}. We start
by establishing that $\bar{x}_j^k$ is within the domain of each function $f_i$
for all $1\leq k<\bar{K}$: fix any $1\leq k<\bar{K}$ and first suppose that
$|\Iback|>0$, which that implies either $j\in\Iback\backslash\Ilip$ or $j$ is
the unique element of $\Iback$. Since $\range(\prox_{\rho f_j})=\dom(f_j)$, we
have that $x_j^k\in\dom(f_j)$ for all $1\leq k<\bar{K}$. Since the domain of a
convex function is a convex set \cite[Proposition 8.2]{bauschke2011convex},
this implies that $\bar{x}_j^k\in\dom(f_j)$. The other possibility is that
$|\Iback|=0$, which means that $j\in\Iforw$ and therefore $\dom(\nabla
f_j)=\calH$, implying that $\dom(f_j)=\calH$, and thus trivially that
$\bar{x}_j^k\in\dom(f_j)$. Finally recall that for $i\neq j$, $i\in
\Ilip\cup\Iforw$, therefore either $\dom(f_i)=\calH$ or $\{x:\|x\|\leq
B_x\}\subseteq \dom(f_i)$. By Lemma \ref{lemBoundXY}, $x_j^k\in \{x:\|x\|\leq
B_x\}$ for all $j=1,\ldots,n$ and $1\leq k<\bar{K}$.  Since the ball $B_x$
convex, we must also have $\bar{x}_j^k\in\dom(f_i)$ in this case.
	
Continuing, we write
	\begin{eqnarray}\label{funcComp}
	\sum_{i=1}^{n}f_i(\bar{x}_{j}^k)-F^* 
	=
	\sum_{i=1}^{n}f_i(\bar{x}_{i}^k)-F^*
	+
	\sum_{i\neq j}(f_i(\bar{x}_{j}^k)-f_i(\bar{x}_i^k)).
	\end{eqnarray}
	The first summation on the right-hand side of this equation is
	$\bigO(1/k)$ by~\eqref{eqMultiFuncResult0}, which has already been
	established.  So we now focus on the terms in second summation.
	
	If $i\in\Iback$ and $i\neq j$, we must have by Assumption \ref{assLip}
	that $i\in\Ilip$. Therefore,
	\begin{align}\label{eqEasLip}
	f_i(\bar{x}_{j}^k)-f_i(\bar{x}_i^k)\leq 
	M_i\|\bar{x}_i^k - \bar{x}_{j}^{k}\| 
	\leq \frac{4 M_i \smallnorm{p^1 - p^*}}{\ualpha k},
	\end{align}
	where the the second inequality arises from \eqref{defergup}.
	 
    Otherwise, we have  $i\in\Iforw$ and $i\neq j$, and since $\dom(\nabla
    f_i) = \calH$, we may use the subgradient inequality to write
	\begin{align}
	f_i(\bar{x}_{j}^k)-f_i(\bar{x}_i^k)
	&\leq
	\langle
	\nabla f_i( \bar{x}_j^k)
	,
	\bar{x}_j^k - \bar{x}_i^k
	\rangle 
	\nonumber\\\nonumber 
		&=
	\langle 
	\nabla f_i( x_i^k)
	,
	\bar{x}_j^k - \bar{x}_i^k
	\rangle
	+
	\langle 
	\nabla f_i( \bar{x}_j^k)-\nabla f_i( x_i^k)
	,
	\bar{x}_j^k - \bar{x}_i^k
	\rangle
	\nonumber\\
	&\leq
	\langle 
	y_i^k
	,
	\bar{x}_j^k - \bar{x}_i^k
	\rangle
	+
      L_i
	\|\bar{x}_j^k - x_i^k\|\|\bar{x}_j^k - \bar{x}_i^k\|
	\nonumber\\
	&\leq
	\| 
	y_i^k
	\|
	\|\bar{x}_j^k - \bar{x}_i^k\|
	+
     L_i(
	\|\bar{x}_j^k\| + \|x_i^k\|)\|\bar{x}_j^k - \bar{x}_i^k\| \nonumber \\
	\label{eqLong}
	&= \Big( 
    \| y_i^k \| + 
	L_i \big( \smallnorm{\bar{x}_j^k} + \smallnorm{\bar{x}_i^k}\big)
	\Big) \smallnorm{\bar{x}_j^k - \bar{x}_i^k},
	\end{align}	
	where the second inequality uses that $y_i^k = \nabla f_i(x_i^k)$ for
	$i\in\Iforw$, the Cauchy-Schwarz inequality, and the Lipschitz continuity
	of $\nabla f_i$.  Lemma~\ref{lemBoundXY} assures us that $\|x_i^k\|,
	\|\bar{x}_j^k\| \leq B_x$ and $\|y_i^k\|\leq B_y$, which in conjunction
	with \eqref{defergup} and $L_i \leq \bar{L}$ leads to
	\begin{align}\label{eqbrb}
	f_i(\bar{x}_{j}^k)-f_i(\bar{x}_i^k)
	\leq
    \frac{4(B_y+2\bar{L}B_x)\|p^1-p^*\|}{\ualpha k}.
	\end{align}

	\noindent 
	Substituting \eqref{eqMultiFuncResult0}, \eqref{eqEasLip}, and
	\eqref{eqbrb} into \eqref{funcComp}, we obtain
	\begin{multline*} 
	\nonumber
	\sum_{i=1}^{n}f_i(\bar{x}_{j}^k)-F^*
	=
	\sum_{i=1}^{n}f_i(\bar{x}_{i}^k)-F^*
	+
	\sum_{i\neq j}\big(f_i(\bar{x}_{j}^k)-f_i(\bar{x}_i^k)\big)
	\\
	\leq 
    \frac{E_2\smallnorm{p^1-p^*}^2 + E_3\smallnorm{p^1-p^*}}{k}
    + 
    \sum_{i\in\Ilip} \frac{4 M_i \smallnorm{p^1 - p^*}}{\ualpha k}
	+
    \frac{4 \lvert\Iforw\rvert B_y(1+2\bar{L}B_x)\|p^1-p^*\|}{\ualpha k}
	\\
	\leq
	\frac{E_2\smallnorm{p^1-p^*}^2 + E_3\smallnorm{p^1-p^*}}{k}
    +
	\frac{4\|p^1-p^*\|}{\underline{\alpha}k}
	\left(
	\sum_{i\in\Ilip}
	M_i
	+
	n(B_y
	+
	2\bar{L}B_x)
	\right)
	.
	\end{multline*} 
\noindent This establishes \eqref{eqMultiFuncResult} with $E_4$ defined in \eqref{defE4}. 
\end{proof}

%\input{asymp1}

%\input{strongMon0}
%!TEX root = convratesv3.tex
\section{Consequences of Strong Monotonicity}\label{secStrong}
We now investigate the consequences of strong monotonicity in
\eqref{probCompMonoMany} for the convergence rate of projective splitting.
\begin{assumption}
	\label{assStrong}
	For \eqref{probCompMonoMany}, there is some $l\in\{1,\ldots,n\}$ for which
	$T_l$ is strongly monotone.
\end{assumption}
\noindent Note that under Assumption \ref{assStrong} it is immediate
that the solution of \eqref{probCompMonoMany} is unique.

%It is straightforward to prove strong convergence, rather than weak convergence, under this assumption. It is also straightforward to compute convergence rates for the algorithm.
\begin{theorem}
Suppose assumptions \ref{assMono}, \ref{assErr}, \ref{assStep}, and \ref{assStrong} hold. Let $z^*$ be the unique solution to \eqref{probCompMonoMany}. Fix $p^* = (z^*,\bw^*)\in\calS$. If $\bar{K}=+\infty$ then 
$x_i^k\to z^*$ for $i=1,\ldots, n$ and $z^k\to z^*$. If $\bar{K}<\infty$, $x_i^{\bar{K}}=z^{\bar{K}+1} = z^*$ for $i=1,\ldots,n$.  Furthermore for all $1\leq k<\bar{K}$
\begin{eqnarray}
\label{eqStrongrte2}
\|x_{\avg}^k - z^*\|^2&\leq& 
\frac{\xi_1\|p^1 - p^*\|^2}{\ubeta\xi_2\mu k}
%\\
%\label{eqStrongrte1}
%	\min_{i=1,\ldots,k}\|x_l^i - z^*\|^2 &\leq& 
%	\frac{\xi_1\|p^1 - p^*\|^2}{\xi_2\mu k}.
\end{eqnarray}
where $x_{\avg}^k = \frac{1}{k}\sum_{i=1}^t x_l^i$ and $\xi_1$ and $\xi_2$ are
defined in \eqref{defxi1} and \eqref{defxi2} respectively.
\end{theorem}
\begin{proof}
Regarding the case $\bar{K} < \infty$, optimality of $x_i^{\bar{K}}$ and
$z^{\bar{K}+1}$ was shown in \cite[Lemma 5]{johnstone2018projective}.

Let $w_n^* = -\sum_{i=1}^{n-1}w_i^*$. Using this notation and recalling the
affine function defined in \eqref{hplane}, we may write
	\begin{align}
		-\varphi_{k}(z^*,{\bf w}^*) &=
		\sum_{i=1}^{n}\langle z^* - x_i^k ,w_i^* - y_i^k\rangle
		\nonumber \\
		&=
		\sum_{i\neq l}
		\langle z^* - x_i^k ,w_i^* - y_i^k\rangle
		+
		\langle z^* - x_l^k ,w_l^* - y_l^k\rangle
		\nonumber \\
		&\geq 		
		\mu\|z^* - x_l^k\|^2, 
		\label{strongMonoIneq}
	\end{align}
where the inequality uses that $(x_i^k,y_i^k)$ and $(z^*,w_i^*)$ are in
$\gra T_i$, along with the monotonicity of $T_i$ for $i\neq l$, and also the
strong monotonicity of $T_l$. Now, Lemma \ref{lemLowerBoundPhi} implies that
$\varphi_k(z^k,\bw^k)\geq 0$. Therefore, for any $1\leq k<\bar{K}$,
\begin{eqnarray*}
	\mu\|z^* - x_l^k\|^2
	&\leq&
	\varphi_k(z^k,{\bf w}^k) - \varphi_k(z^*,{\bf w}^*)
	\\
	&\overset{(a)}{=}&
	\langle \nabla \varphi_k,p^k - p^*\rangle
	\\
		&\overset{(b)}{=}&
	\frac{1}{\alpha_k}\langle
	p^k - p^{k+1}
	,
	p^k - p^*
	\rangle
	\\
		&\overset{(c)}{=}&
	\frac{1}{2\alpha_k}
	\left(
	\|p^k - p^{k+1}\|^2
	+
	\|p^k - p^*\|^2
	-
	\|p^{k+1}-p^*\|^2
	\right)
	\\
		&\overset{(d)}{\leq}&
	\frac{1}{2\ualpha}
	\left(
	\|p^k - p^{k+1}\|^2
	+
	\|p^k - p^*\|^2
	-
	\|p^{k+1}-p^*\|^2
	\right).
\end{eqnarray*}
Above, (a) uses that $\varphi_k$ is affine, (b) employs \eqref{eqUpds} from
Lemma \ref{propFejer}, (c) uses \eqref{eqCosine}, and (d) uses Lemma
\ref{lb_ratio} and \eqref{eqFejer} from Lemma \ref{propFejer}. Summing the resulting
inequality yields, for all $1\leq k<\bar{K}$,
\begin{eqnarray}
	\mu\sum_{t=1}^{k}
	\|z^* - x_l^t\|^2
	&\leq&
	\frac{1}{2\ualpha}
	\left(
	\sum_{t=1}^k\|p^t - p^{t+1}\|^2
	+
	\|p^1 - p^*\|^2
	\right)
	%\nonumber\\
	%&
	\leq
	%&
    \frac{\xi_1}{\ubeta\xi_2}\|p^1 - p^*\|^2, \label{eqIsSum}
\end{eqnarray}
where the second inequality uses Lemma \ref{propFejer} and
\eqref{eqSums} from Lemma \ref{lb_ratio}. So, when $\bar{K}=\infty$, we have
$x_l^k\to z^*$. Lemma \ref{lemSum} asserts that 
$
\nabla\varphi_k\to 0
$ 
in the $\bar{K}=\infty$ case, so from \eqref{defGrad} we conclude that
$x_i^k\to z^*$ for $i=1,\ldots,n$. Lemma \ref{lemSum} also establishes that
$\varphi_k(p^k)\to 0$, so by Lemma \ref{lemLowerBoundPhi}, we have 
% \begin{align*}
% (\forall\,i=1,\ldots,n) \qquad
$
\|z^k - x_i^k\| \to 0
$
for all $i=1,\ldots,n$
% \end{align*}
and hence
$
z^k \to z^*.
$
Finally, the convexity of the function $\|\cdot\|^2$ in conjunction with
\eqref{eqIsSum} yields \eqref{eqStrongrte2}.
% and \eqref{eqStrongrte1} follows since the smallest of a finite set of
% numbers must be upper bounded by the average.
\end{proof}

%!TEX root = convratesv3.tex
\section{Linear Convergence Under Strong Monotonicity and Cocoercivity}
\label{secLin}
If we introduce a cocoercivity assumption along with strong monotonicity, we
can derive linear convergence. We require that all but one of the operators be
cocoercive. Without loss of generality, we designate $T_n$ to be be the
operator that need not be cocoercive.

%We also require Assumption \ref{assStrong}: at least one of the operators is strongly monotone. 

\begin{assumption}\label{assCoc}
In \eqref{probCompMonoMany}, suppose that for $i=1,\ldots,n-1$, the operator
$T_i$ is cocoercive with parameter $\Gamma_i>0$. Let $\bar{\Gamma} =
\max_{i=1,\ldots,n-1} \Gamma_i$.
\end{assumption}

\noindent If this assumption holds, then $T_1,\ldots,T_{n-1}$ are all
single-valued.   If Assumption~\ref{assStrong} also holds, this
single-valuedness implies that not only is the solution $z^*$
to~\eqref{probCompMonoMany} unique, but that the extended solution set $\calS$
must be singleton of the form $\big\{(z^*,w_1^*,\ldots,w_{n-1}^*)\big\} =
\big\{(z^*,T_1z^*,\ldots,T_{n-1}z^*)\big\}$.

\begin{theorem}\label{ThmLinear}
Suppose assumptions \ref{assMono}, \ref{assErr}, \ref{assStep}, \ref{assStrong}, and
\ref{assCoc} hold. Let $z^*$ be the unique solution to
\eqref{probCompMonoMany} and take $p^*=(z^*,\bw^*) =
(z^*,T_1z^*,\ldots,T_{n-1}z^*) \in\calS$. If $\bar{K}<\infty$, $p^{\bar{K}+1}
= p^*$. For all $1\leq k<\bar{K}$,
\begin{align}\label{eqLinConv}
\|p^{k+1}-p^*\|^2\leq (1-E_5)\|p^k-p^*\|^2
\end{align}
where
\begin{align}
E_5
=
\frac{1}{2}
\left(
\frac{8\xi_1^2(1+\gamma)^2\max\{\mu^{-1},\bar{\Gamma}\}^2
+
2\gamma \xi_1
}{\ubeta^2\xi_2^2}
+
2 E_1
\right)^{-1}
\in\left]0\,,1/4\right].\label{defXi3}
\end{align}

\end{theorem}

\begin{proof}
		For the finite-termination case, optimality of $p^{\bar{K}+1}$ was established
in \cite[Lemma 5]{johnstone2018projective}.  Henceforth, we thus assume that
$\bar{K} = \infty$.
The key idea of the proof is to show that for $E_5$ defined in \eqref{defXi3},
\begin{eqnarray}\label{eqKeyBound}
E_5\|p^k - p^*\|^2\leq \|p^{k+1}-p^k\|^2,
\end{eqnarray} 
which, when used with \eqref{eqFejer} of Lemma \ref{propFejer}, implies
\begin{eqnarray*}
\|p^{k+1}-p^*\|^2
&\leq&
\|p^k - p^*\|^2
-
\|p^{k+1}-p^k\|^2
\leq
(1-E_5)\|p^k - p^*\|^2 .
\end{eqnarray*}
We now establish (\ref{eqKeyBound}).
% From Assumption \ref{assStrong}, $T_l$ is strongly monotone. From Assumption
% \ref{assCoc}, $T_i$ is cocoercive for $i=1,\ldots,n-1$. Furthermore recall
% that $\varphi_k(p^k)\geq 0$ for all $1\leq k<\bar{K}$. Using these facts along
% with regular monotonicity
For $1\leq k<\bar{K}$, we have $\varphi_k(p^k)\geq 0$ and hence
\begin{align*}
	\varphi_k(p^k) - \varphi_k(p^*)
	&\geq
\sum_{i=1}^{n}
\langle z^* - x_i^k,w_i^* - y_i^k\rangle
\\	
&=
\frac{1}{2}
\sum_{i=1}^{n}
\langle z^* - x_i^k,w_i^* - y_i^k\rangle
+
\frac{1}{2}
\sum_{i=1}^{n}
\langle z^* - x_i^k,w_i^* - y_i^k\rangle
\\
	&\geq 
	\frac{\mu}{2}\|z^* - x_l^k\|^2
	+
	\sum_{i=1}^{n-1} 
	\frac{1}{2\Gamma_i}
	\|y_i^k - w_i^*\|^2.
\end{align*}
The last inequality here follows because $\sum_{i=1}^{n}
\langle z^* - x_i^k,w_i^* - y_i^k\rangle \geq \mu\|z^* - x_l^k\|^2$ by the
same logic as in~\eqref{strongMonoIneq}, and by the cocoerciveness of
$T_1,\ldots,T_{n-1}$. We therefore obtain
\begin{align}
\frac{\mu}{2}\|z^* - x_l^k\|^2
+
\sum_{i=1}^{n-1} 
\frac{1}{2\Gamma_i}
\|y_i^k - w_i^*\|^2
%&
\leq
\varphi_k(p^k) - \varphi_k(p^*)
&=
\langle\nabla\varphi_k,p^k - p^*\rangle 
\nonumber\\
&=
\frac{1}{\alpha_k}\langle p^k - p^{k+1},p^k - p^*\rangle,
\label{return2}
\end{align}
where we have again used that $\varphi_k$ is affine, and also \eqref{eqUpds}.
Continuing, we use Young's inequality applied to~\eqref{return2} to write, for
any $\nu>0$,
\begin{eqnarray}
\frac{\mu}{2}\|z^* - x_l^k\|^2
+
\sum_{i=1}^{n-1} 
\frac{1}{2\Gamma_i}
\|y_i^k - w_i^*\|^2
&\leq&
\frac{1}{2\nu\ualpha^2}\|p^{k+1}-p^k\|^2
+
\frac{\nu}{2}\|p^k - p^*\|^2,
\nonumber 
\end{eqnarray}
which implies
\begin{align}
\|z^* - x_l^k\|^2
+
\sum_{i=1}^{n-1} 
\|y_i^k - w_i^*\|^2
\leq 
\frac{\xi_5}{\nu\ualpha^2}\|p^{k+1}-p^k\|^2
+
\xi_5 \nu\|p^k - p^*\|^2,
\label{bad}
\end{align}
where $\xi_5=\max\{\mu^{-1},\bar{\Gamma}\}$.
From Lemma \ref{lemBasic}, we then deduce that
\begin{align}
\|p^k - p^*\|^2 
&=
\gamma\|z^k - z^*\|^2
+
\sum_{i=1}^{n-1} 
\|w_i^k - w_i^*\|^2
\nonumber\\
&\leq
2\gamma\|z^*-x_l^k\|^2
+
2\gamma\|x_l^k - z^k\|^2
\nonumber\\
&\qquad
+\;
2
\sum_{i=1}^{n-1} 
\|y_i^k - w_i^*\|^2
+
2
\sum_{i=1}^{n-1} 
\|y_i^k - w_i^k\|^2.
\nonumber\\
&\leq
2(1+\gamma)
\left(
\|z^*-x_l^k\|^2
+
\sum_{i=1}^{n-1} 
\|y_i^k - w_i^*\|^2
\right)
+
2\gamma\|x_l^k - z^k\|^2
% \nonumber\\
% &
+
2
\sum_{i=1}^{n-1} 
\|y_i^k - w_i^k\|^2.
\label{mac}
\end{align}
Substituting the upper bound (\ref{bad}) for the term in parentheses in
(\ref{mac}), and then using \eqref{eqLast1} and \eqref{eqLast2} from Lemma
\ref{lemSum} on the other two terms, we conclude that
\begin{align}
\|p^k-p^*\|^2
&\leq 
2(1+\gamma)
\left(
\frac{\xi_5}{\nu\ualpha^2}\|p^{k+1}-p^k\|^2
+
\xi_5 \nu \|p^k - p^*\|^2
\right)
+\frac{2\gamma \xi_1}{\ubeta^2\xi_2^2}\|p^{k+1}-p^k\|^2
\nonumber\\\nonumber
&\qquad
+\;
2E_1\|p^{k+1}-p^k\|^2.
\end{align}
Rearranging this inequality yields
\begin{align}
(1-2\nu(1+\gamma)\xi_5)\|p^k - p^*\|^2\leq 
\left(
\frac{2(1+\gamma)\xi_5}{\nu\ualpha^2}
+\frac{2\gamma \xi_1}{\ubeta^2\xi_2^2}
+
2 E_1
\right)
\|p^{k+1}-p^k\|^2.  \label{godzilla}
\end{align}
We now set
\begin{align*}
\nu=
\frac{1}{4(1+\gamma)\xi_5}.
\end{align*}
Using this value of $\nu$ in~\eqref{godzilla} implies \eqref{eqKeyBound} with
\begin{align}\label{defXiHat2}
E_5
=
\frac{1}{2}
\left(
\frac{8(1+\gamma)^2(\xi_5)^2}{\ualpha^2}
+\frac{2\gamma \xi_1}{\ubeta^2\xi_2^2}
+
2 E_1
\right)^{-1},
\end{align}
which reduces to the expression in \eqref{defXi3}. Considering \eqref{defXi3},
we note that since $\xi_2>0$, $\ubeta>0$, $\mu>0$, and $E_1<\infty$, we must
have $E_5>0$.

%\input{asymp2}
%!TEX root = convratesv3.tex
% If $E_5>1$, then \eqref{eqLinConv} does not make sense. If $E_5=1$ then
% \eqref{eqLinConv} suggests the algorithm may converge in one iteration. 
Finally, we show that $E_5 \leq 1/4$ as claimed in~\eqref{defXi3}
(in particular, this precludes the nonsensical situation that $E_5 > 1)$.
% We now
% show that neither of these situations is possible as $E_5$ is bounded above by
% $1/4$. 
We write
\begin{align}
\label{rewrite}
E_5
&=
\frac{1}{2}
\left(
\frac{8\xi_1^2(1+\gamma)^2\max\{\mu^{-1},\bar{\Gamma}\}^2
	+
	2\gamma \xi_1
}{\ubeta^2\xi_2^2}
+
2 E_1
\right)^{-1}
\\\nonumber
&\overset{(a)}{\leq} 
\frac{
\xi_2^2
}
{
\gamma\xi_1
}
\\\nonumber
&=
\frac{
\min
\left\{
(1-\sigma)\overline{\rho}^{-1}
,\;
\min_{j\in\Iforw}
\left\{
\overline{\rho}_{j}^{-1} - L_j 
\right\}
\right\}^2
}
{
2n\gamma\left[1+2\gamma^{-1}\left(\bar{L}^2|\Iforw|+\underline{\rho}^{-2}(1+\delta)\right)\right]
}
\\\nonumber
&\leq
\frac{
	\min
	\left\{
	(1-\sigma)\overline{\rho}^{-1}
	,\;
	\min_{j\in\Iforw}
	\left\{
	\overline{\rho}_{j}^{-1} - L_j 
	\right\}
	\right\}^2
}
{
	4n\underline{\rho}^{-2}(1+\delta)
}
\\\nonumber
&\overset{(b)}{\leq}
\frac{
\left(
	(1-\sigma)\overline{\rho}^{-1}
+
	\sum_{j\in\Iforw}
	\left(
	\overline{\rho}_{j}^{-1} - L_j 
	\right)
	\right)^2
}
{
	4n\underline{\rho}^{-2}(1+\delta)(|\Iforw|+1)^2
}
\\\nonumber
&\overset{(c)}{\leq}
\frac{
	\overline{\rho}^{-2}
	+
	\sum_{j\in\Iforw}
	\overline{\rho}_{j}^{-2}
}
{
	4n\underline{\rho}^{-2}(1+\delta)(|\Iforw|+1)
}
\overset{(d)}{\leq}
\frac{
1
}
{
4n(1+\delta)
}
\leq\frac{1}{4},
\end{align}
where we employ the following reasoning: first, (a) uses that all terms within
the parentheses in \eqref{rewrite} are positive and that $\ubeta<2$.  In (b),
we use that the minimum of a set of numbers cannot exceed its average.
Inequality (c) follows by observing that $(1-\sigma)\leq 1$ and $-L_j\leq 0$,
and then using Lemma \ref{lemBasic}. Finally, (d) uses that $\urho\leq
\orho_j$ and $\urho\leq\orho$.
\end{proof}

%!TEX root = convratesv3.tex
\section{Simplifications when $n=1$}
\label{secnis1}
Suppose $n=1$, in which case we have either $1\in \Iback$ or $1\in\Iforw$. In
either case, using the convention discussed at the beginning of Section
\ref{subsecNote}, we have $w_1^k = 0$ for all $k$ and the affine function
defined in \eqref{hyper} becomes
\begin{align*}
\varphi_k(p) =\varphi_k(z) = \langle z-x^k,y^k\rangle,
\end{align*}
where we have written $x_1^k = x^k$ and $y_1^k = y^k$. 

If $1\in\Iback$, using $w_1^k = 0$, in the update on lines
\ref{lineaupdate}-\ref{lineBackwardUpdateY} of the algorithm yields
\begin{align*}
x^k + \rho^k y^k = z^k \qquad \text{ and } \qquad y^k\in T x^k,
\end{align*}
where we have written $T_1=T$ and $\rho_1^k = \rho^k$. 
This implies that
$
\varphi_k(z^k) =  \rho^k\|y^k\|^2.
$
Furthermore,
$
\nabla_z\varphi_k  = \gamma^{-1}y^k
$
and so $\|\nabla_z\varphi_k\|^2 = \gamma\cdot\gamma^{-2}\|y^k\|^2 = \gamma^{-1}\|y^k\|^2$. 
Therefore, using \eqref{eqUpds}, we have
\begin{align*}
z^{k+1} = z^k - \frac{\beta_k\varphi_k(p^k)}{\|\nabla\varphi_k\|^2}\nabla_z\varphi_k 
&=
z^k -\beta_k\rho^k y^k \\ &=
(1-\beta_k)z^k + \beta_k x^k =
(1-\beta_k)z^k
+
\beta_k\prox_{\rho^k}(z^k).
\end{align*}
Thus, when $n=1$ and $1\in\Iback$, projective splitting reduces to the relaxed
proximal point method of~\cite{EckBer92}; see also \cite[Theorem
23.41]{bauschke2011convex}. In fact, when one allows for approximate
evaluation of resolvents, projective splitting with $n=1\in\Iback$ reduces to
the hybrid projection proximal-point method of Solodov and Svaiter 
\cite[Algorithm 1.1]{solodov1999hybrid}. However, the error criterion
of~\cite[Eq.~(1.1)]{solodov1999hybrid} is more restrictive than the
conditions~\eqref{err1}--\eqref{err2} that we propose in
Assumption~\ref{assErr}.

On the other hand, if $1\in\Iforw$, considering lines
\ref{ForwardxUpdate}-\ref{ForwardyUpdate} or the algorithm with $w_1^k = 0$
yields
\begin{align*}
x^k = z^k - \rho^k T z^k \qquad \text{ and } \qquad y^k = T x^k.
\end{align*} 
Furthermore, we have 
\begin{align*}
\varphi_k(z^k) = \rho^k \langle T z^k,T x^k\rangle 
\qquad \text{ and } \qquad 
\nabla_z \varphi_k = \gamma^{-1}T x^k.
\end{align*}
Therefore,
\begin{align*}
z^{k+1} 
&= 
z^k 
-
 \frac{\beta_k\varphi_k(p^k)}{\|\nabla\varphi_k\|^2}\nabla_z\varphi_k 
 =
 z^k - 
 \frac{\beta_k\rho^k \langle T z^k,T x^k\rangle }
 { \|Tx^k\|^2}
 Tx^k.
\end{align*}
Thus, the method reduces to
\begin{align*}
x^k &= z^k - \rho^kTz^k
\\
z^{k+1} &= z^k - \tilde{\rho}^k Tx^k,
\end{align*}
where
\begin{align}\label{eqStepAd}
\tilde{\rho}^k= \frac{\beta_k\rho^k \langle T z^k,T x^k\rangle}{\|Tx^k\|^2}.
\end{align}
This is the unconstrained version of the extragradient method
\cite{korpelevich1977extragradient}. When $\beta_k=1$, the stepsize
\eqref{eqStepAd} corresponds to the extragradient stepsize proposed by Iusem
and Svaiter in \cite{iusem1997variant}.  Furthermore, the backtracking
linesearch for the extragradient method also proposed in
\cite{iusem1997variant} is almost equivalent in the unconstrained case to the
linesearch we proposed for processing individual operators within projective
splitting in \cite{johnstone2018projective}, except that Iusem and Svaiter use
a more restrictive termination condition (perhaps necessary because they also
account for the constrained case of the extragradient method).

While these observations may be of interest in their own right, they also have
implications for the convergence rate analysis. In particular, that projective
splitting reduces to the proximal point method suggests that the $\bigO(1/k)$
ergodic convergence rate for \eqref{ProbOpt} derived in Section \ref{secFunc}
cannot be improved beyond a constant factor. This is because the same rate is
unimprovable for the proximal point method under the assumption that the
stepsize is bounded from above and below \cite{guler1991convergence}.
\bibliographystyle{spmpsci}
\bibliography{refs}

\end{document}